\documentclass[12pt]{article}
\usepackage{philstyle}

\def\cF{{\cal F}}

\def\QED{\ $\blacksquare$\smallskip}

\def\de#1{\textit{#1}}

\def\rationals{{\mathbb Q}}
\def\reals{{\mathbb R}}
\def\integers{{\mathbb Z}}
\def\naturals{{\mathbb N}}
\def\raw{\rightarrow}
\def\Fix{\text{Fix}}

\def\us{s}

\def\ut{t}

\def\I{^{-1}}

\def\vn{\vec{n}}
\def\vm{\vec{m}}

\def\vp{\vec{p}}

\def\tf{\tilde{f}}
\def\tg{\tilde{g}}
\def\tM{\tilde{M}}
\def\tphi{\tilde{\phi}}
\def\tpsi{\tilde{\psi}}

\def\homeo{homeomorphism}
\def\homeos{homeomorphisms}

\def\Z{\integers}

\def\Q{\rationals}
\def\N{\naturals}
\def\R{\reals}

\DeclareMathOperator{\id}{id}
\DeclareMathOperator{\im}{im}

\def\tcF{\tilde{\cal F}}
\def\pA{pseudo-Anosov}

\DeclareMathOperator{\spec}{spec}

\DeclareMathOperator{\rot}{rot}

\DeclareMathOperator{\rank}{rank}

\def\tphi{{\tilde{\phi}}}
\def\tcF{\tilde{\cal F}}
\def\tgamma{{\tilde{\gamma}}}

\def\tx{{\tilde{x}}}
\def\ty{\tilde{y}}
\def\tz{\tilde{z}}

\DeclareMathOperator{\SL}{SL}

\def\osum{\oplus}

\def\ie{i.e.}
\def\cf{\textit{cf.}}
\def\tomega{{\tilde{\omega}}}

\def\cL{\tilde{L}}
\def\tL{\tilde{\cal{L}}}

\def\hT{\hat{T}}

\def\hh{\hat{h}}
\def\hht{\hat{t}}

\def\gen#1{\langle #1 \rangle}
\def\genp#1{\langle #1 \rangle_+}
\def\barg{\overline{g}}
\def\hatg{\hat{g}}
\def\seto{\{1\}}
\def\hS{\hat{S}}
\def\utn{\ut^{(0)}}

\def\utm{\ut^{(m)}}
\def\hL{\hat{L}}
\def\tR{\tilde{R}}
\def\talpha{\tilde{\alpha}}
\def\tbeta{\tilde{\beta}}
\def\hus{\hat{\us}}
\def\hSigma{\hat{\Sigma}}
\def\hsigma{\hat{\sigma}}

\def\td{\tilde{d}}

\usepackage{paralist}
\setdefaultleftmargin{2\parindent}{}{}{}{}{}
\setdefaultenum{(a)}{}{}{}

\begin{document}

\begin{center}
\textbf{\sectionfont
Transitivity of Surface Dynamics Lifted to Abelian Covers}\\
\medskip
 Philip Boyland \\
Dept. of Mathematics\\
University of Florida\\
   Gainesville, FL 32611-8105 \\ 
\end{center}

\medskip

\noindent\textbf{\sectionfont Abstract:}
 A \homeo\ $f$ of a manifold $M$ is called $H_1$-transitive if
there is a transitive lift of an iterate of $f$ to the universal
Abelian cover $\tM$. Roughly speaking, this means that $f$ has
orbits which repeatedly and densely explore all elements of $H_1(M)$.
For a rel \pA\ map $\phi$ of a compact surface $M$ we show that the following
are equivalent: (a) $\phi$ is  $H_1$-transitive, (b) the action
of $\phi$ on $H_1(M)$ has spectral radius one, and (c) the lifts
of the invariant foliations of $\phi$ to $\tM$ have dense leaves. 
The proof relies on a characterization of transitivity for twisted
$\Z^d$-extensions of a transitive subshift of finite type.
\medskip

\section{Introduction}
There are many ways to characterize the complexity of a dynamical
system  on a manifold $M$.
 In this paper we focus on the characterization of \de{$H_1$-transitivity}.
A \homeo\ $f$ is called $H_1$-transitive when, roughly speaking, 
it has orbits which repeatedly and densely explore all the elements
of first homology. As is natural and commonly done,  
we formalize this notion  by ``unwraping'' the 
manifold by passing to a covering space. For $H_1$-transitivity
the appropriate lift is to the universal Abelian cover $\tM$
which is the covering space whose automorphism group is equal to $H_1(M;\Z)$. 
Translations of  orbits lifted to  $\tM$ correspond 
to motion around homologically nontrivial loops in $M$.
Thus we adopt the definition:
\begin{definition}
A \homeo\ $f:M\raw M$ is called $H_1$-transitive if there
is a lift $\tg$ 
of an iterate of $f$ to the universal Abelian cover $\tM$ such
that $\tg$ has an orbit which is dense in $\tM$.
\end{definition}

Our main concern here is with a particular class of maps, rel \pA\ \homeos\
of surfaces. These maps are an essential piece of 
Thurston's classification of isotopy classes of surface
\homeos\ and have many nice dynamical properties including
a symbolic description by a transitive subshift of finite type.
Rel \pA\ maps are characterized by the existence
of a transverse pair of (mildly)
singular, invariant foliations, each equipped with a transverse measure
which expands or contracts under the map. Every nontrivial leaf of these
foliations is dense in the surface $M^2$. For this class of
maps we have the following 
equivalence between (a) the dynamical property of $H_1$-transitivity,
(b) an algebraic condition on the spectral radius $\rho(\phi_*)$ 
of the induced action of the  map $\phi$ on $H_1(M^2)$, 
 and (c) a topological condition
on the invariant foliations when lifted to the universal
Abelian cover.
\newpage
\begin{theorem}\label{pA2}
Assume that  $\phi:M^2\raw M^2$ is a rel \pA\ map.
The following are equivalent: 
\begin{compactenum}
 \item $\phi$ is $H_1$-transitive, 
 \item $\rho(\phi_*) = 1$,
 \item There is a leaf of the lifted foliation $\tcF^u$ 
 which is dense in the universal Abelian cover $\tM$.
 \end{compactenum} 
\end{theorem}
The proof of this theorem depends on Theorem~\ref{pA1} 
which gives a number of conditions which are equivalent
to the total transitivity (all iterates are transitive)
 of a lifted rel \pA\ map $\tphi$. These conditions
include that $\tphi$ is topologically mixing, that 
the periodic orbits of $\tphi$ are dense in $\tM$, and that
$\rho(\phi_*) = 1$ coupled with a condition on the rotation set
of the Fried quotient of $(\tphi, M^2)$. A version of Theorem~\ref{pA1}
for the annulus was in \cite{bgh} and for
the torus in \cite{parwanithesis}. 

The proof of Theorem~\ref{pA1} in turn depends on
Theorem~\ref{MST} which characterizes total transitivity 
of a twisted skew product with group factor $\Z^d$ 
over a base subshift of finite type $(\Sigma, \sigma)$.
The twisted skew products considered here 
 are maps $\tau:\Sigma\times\Z^d \raw \Sigma\times\Z^d$ of the form 
\begin{equation*}
\tau(\us, \vn) = (\sigma(\us), \Phi(\vn) + h(\us)),
\end{equation*}
 where $\Phi:\Z^d\raw
\Z^d$ is  the \de{twisting isomorphism} or just the
\de{twisting}, 
and $h:\Sigma\raw \Z^d$ is the \de{height function}.
When the twisting is trivial ($\Phi = \id$) the map
$\tau$ is called an \de{untwisted skew product} or
just a \de{skew product}. The ergodic theory and topological dynamics
 of untwisted skew products with various bases and group
components have 
been intensely studied for at least 50 years 
(see \cite{parrypollicott} for some history), and their use  as symbolic
models for dynamics lifted to covering spaces is well
established. The transitive untwisted skew products over  
subshifts of finite type with group factor $\Z^d$ 
were characterized by Coudene (\cite{coudene})
and those with group factor $\R^d$ by Nitica (\cite{nitica}). 
Twisted skew products are themselves a special case of the well-studied 
notion of a group extension (see, for example, \cite{mentzen}).

When a twisted skew product models 
the lift of a rel \pA\ map $\phi$ to the universal Abelian
cover, the twisting
isomorphism is the action of $\phi$ on
first homology, or $\Phi = \phi_*$. Thus to study
maps which do not act trivially on homology
we must consider the case of nontrivial twisting.
The first observation in this study is that the 
 coarse connection between
the dynamics of a lift $\tphi$ to $\tM$ and the action of $\phi_*$ on 
$H_1(M;\R)$ 
implies that when $\rho(\phi_*) > 1$, there will be  open sets in $\tM$
whose iterates under $\tphi$ will go to infinity 
exponentially fast (see \eqref{geo} below). Thus
$\rho(\phi_*) = 1$ is a necessary condition for transitivity
(the case $\rho(\phi_*) < 1$ cannot occur because $\phi$ 
is a \homeo).

The next obvious necessary condition for the transitivity
of $\tphi$ can be informally expressed by
``orbits of $\tphi$ must go to infinity in all directions''.
 Gottschalk and Hedlund were the first to
notice that this  condition can also be sufficient for
transitivity (\cite{GH}).
Perhaps the most common way to formalize the ``all directions''
condition for an untwisted skew product is to use
the collection of  displacements, $D(\tau)$, of
periodic orbits of the base map (see \S\ref{height}). 
Here we primarily use the asymptotic
average  displacements as they are
more tractable under the iterations and translations of 
maps required here. We maintain the usual
terminology from lifted dynamics and call
this average displacement the  \de{rotation vector} of
an orbit. The set of all rotation vectors is the 
\de{rotation set}, and we formulate the ``all
directions'' condition in Theorem~\ref{rotthm} by requiring  that
 zero be in the interior of
the rotation set. This condition alone does not imply
transitivity in the untwisted case, but requires the
addition of a condition,  the
\de{finite lifting property},  which is the analog of
the fact that the lift of  a \pA\ map
 to any compact covering space is transitive.

The definition of the displacement set or rotation vector
requires that each point in the base be assigned
a well-defined displacement in the group factor or
cover. This is only possible when the
skew product is untwisted or $\phi_* = \id$ (see Remark~\ref{diff}).
Thus the weaker hypothesis of $\rho(\phi_*) = 1$
requires additional consideration. The first
step is to note that 
 a classic theorem of Kronecker implies that for some $N>0$, 
$\spec(\phi_*^N) = \seto$, where $\spec$ indicates
the spectrum (see comment above Definition~\ref{Ndef}).
 
We call an isomorphism $\Psi$ with $\spec(\Psi) = \seto$ 
a \de{generalized shear},  
because  over the reals one can find a basis in
which $\Psi$  is represented by its Jordan matrix
of  ones on the diagonal and perhaps also
some ones on the super diagonal. However, in general one cannot
conjugate to this Jordon form in $\SL(d, \Z)$ and so 
we use instead the form given in Lemma~\ref{CFL}.
Using the basis of $\Z^d$ given by Lemma~\ref{CFL},
we can treat a $\tau$ with
generalized shear twisting as a sequence of untwisted
skew products over countable state Markov shifts and
so prove transitivity by induction. The $n=0$ step
of this argument requires the transitivity of
the largest quotient on which $\tau$ is untwisted.
The transitivity of this Fried quotient is obtained
using Theorem~\ref{rotthm} and the main induction step 
is handled by Lemma~\ref{MIL}.

In the last section of the paper after proving Theorem~\ref{pA2}
we comment in \S\ref{idcase} on some dynamical properties of the 
$\phi_* = \id$ case and in Proposition~\ref{folinfo}
 on some topological properties of the lifted foliations.

\section{Topological Transitivity and Countable State Markov Shifts}

In this paper $M$ is always a compact, orientable surface perhaps
with boundary. 
All \homeos\ $h:M\raw M$ are orientation preserving. 
If there is no coefficient ring given,
homology is always  with integer coefficients, 
and so $H_1(M) = H_1(M,\Z)\cong\Z^d$ for some $d\in\N$.

For a topological space $X$, the closure, interior and
frontier are denoted $Cl(X)$, $Int(X)$, and $Fr(X)$, respectively.
For any map $f$, its image is denoted $im(f)$.
For a \homeo\ $h$ of $X$, the \de{forward orbit} of a 
point $x\in X$ is
$o_+(x,h) := \{  x, h(x), h^2(x),\dots \}$, 
the \de{backward orbit} 
is $o_-(x,h) := \{ \dots, h^{-2}(x), h^{-1}(x), x \}$, and
the \de{orbit} is $o(x,h) = o_+(x,h) \cup o_-(x,h)$.

\subsection{Transitivity and mixing}
A \homeo\ $h$ is called \de{topologically transitive} or
just \de{transitive}
if for every pair of open sets $U_1$ and $U_2$ there is
an integer $n\in\Z$ with $h^n(U_1)\cap U_2 \not=\emptyset$, and
$h$ is called \de{topologically mixing}
if for every pair of open sets $U_1$ and $U_2$ there is
an integer $N\in\Z$ with $h^n(U_1)\cap U_2 \not=\emptyset$ for
all $n\geq N$.  If $h^n$ is transitive for all $n>0$, then 
$h$ is called \de{totally transitive}. 

It is standard that  
 if $h^n$ is transitive or topologically mixing for some
$n>0$, then $h$ has the same property. 
 Further, topologically mixing implies totally transitive
and in many cases is equivalent to it.
A standard result on transitivity is
\begin{lemma}\label{bairetrans}
Assume that $h$ is a \homeo\ of a complete, separable 
metric space $X$. The following are equivalent:
\begin{compactenum}
\item $h$ is transitive,
\item There exists an orbit of $h$ that is dense in
$X$, $Cl(o(x,h)) = X$,
\item There is a dense, $G_\delta$-subset $Y\subset X$ so
that $y\in Y$ implies that $Cl(o_+(y,h)) = X$
and $Cl(o_-(y,h)) = X$.
\end{compactenum}
\end{lemma}

\subsection{Countable state Markov shifts}\label{Markovdef}
We recall the basic definitions and properties of
countable state Markov shifts. For more details see
\cite{kitchens}.
A countable state Markov shift is built from a countable
set of \de{states} $S = \{1, 2, 3, \dots\}$.
 The \de{transition matrix} $C$ is
indexed by $S\times S$, and  the entries of $C$  are
all $0$ or $1$ and are denoted $C_{i,j}$. 
An \de{allowable one-step transition} for $C$ is
a pair of states $s_1$ and $s_2$ so that 
$C_{s_1, s_2} = 1$. An \de{allowable $n$-step transition} or
 \de{allowable $n$-block}
is a  list of $n$-states $s_1 s_2 \dots s_n$ with
each pair $s_i s_{i+1}$ an allowable one-step transition.
In general, if there is an allowable transition of any
length between two states $a$ and $b$, we say
that there is is an allowable transition between $a$ and
$b$, and this situation is denoted $a\raw b$. 

The collection of bi-infinite sequences of states is $S^\Z$ and
the \de{shift space}, $\Sigma$, defined by the matrix
$C$ is the subspace of sequences 
all of whose finite blocks represent  transitions
which are allowed by the matrix $C$, or equivalently,
\begin{equation*}
\Sigma = \{ \us\in S^\Z : C_{s_i, s_{i+1}} = 1\ \text{for all}\ i\in\Z\}.
\end{equation*}
The \de{shift map} on $\Sigma$ is the left shift $\sigma:\Sigma\raw\Sigma$.
The shift space and the shift together constitute the
countable state Markov shift which is denoted $(\Sigma,\sigma)$.
When $S$ is a finite set, $(\Sigma,\sigma)$ is called a
\de{subshift of finite type} or a \de{topological Markov chain}.

A metric on $S$ which gives it the discrete topology yields
 a product metric on
$S^\Z$ and  a subspace metric induced on $\Sigma$. Under the 
resulting topology
$\Sigma$ is a totally disconnected, separable, complete metric
space and the shift map
is a \homeo. If $S$ is finite, $\Sigma$ is compact, 
and otherwise it is not. All the transition matrices here
will have finite row and sum columns which implies that
$\Sigma$ is locally compact.

The cylinder sets form a countable base for the topology of $\Sigma$.
A \de{central cylinder set of length two} is 
\begin{equation}\label{cyldef}
[a, b]_0 := \{ \us\in\Sigma : s_0 = a\ \text{and}\ s_1 = b\}
\end{equation} 
where $a\raw b$ is an
allowable  one step transition for $(\Sigma,\sigma)$.

The standard characterization of  transitive 
subshifts of finite types also holds for countable state Markov shifts: 
the system $(\Sigma,\sigma)$ is topologically
transitive if and only if it is irreducible, \ie\ for any pair of states 
$a_1$ and $a_2$ there is an allowable transition  $a_1\raw a_2$  
Also, if $(\Sigma,\sigma)$ is transitive and has a fixed point,
then it is topologically mixing  (see Observation 7.2.2 in \cite{kitchens}).
If   $(\Sigma,\sigma)$ totally transitive, then as just noted,
it certainly has a periodic point, say of period $n$, and so
$\sigma^n$ is topologically mixing, and so $\sigma$ is also.
Thus for countable state Markov shifts,
totally transitive is equivalent to topologically mixing.

\section{Abelian Covering Spaces and the Rotation set}
 Recall that regular, connected
covering spaces of the surface $M$ are in one-to-one correspondence with
normal subgroups $G \triangleleft\pi_1(M)$. For the cover corresponding
to $G$, the \de{deck group}  (also called the group of cover automorphisms)
is naturally identified with the quotient  $\pi_1(M)/G$. 
We shall be exclusively concerned with \de{Abelian covers}.
 These are the covering spaces for which the deck group 
is Abelian. The largest such cover corresponds
to $G = [\pi_1(M), \pi_1(M)]$, the commutator subgroup, in
which case the deck group is 
$ H_1(M)$. This covering space
is called the \de{universal Abelian cover} and is denoted
here as $\tM$. Any other Abelian cover can be obtained
by moding out $\tM$ by the action of a subgroup  
$\Gamma\subset H_1(M)$. This quotient is denoted
 $ \tM_\Gamma := \tM/\Gamma$. Note that $\tM$ is a cover over
$\tM_\Gamma$ with deck group $\Gamma$ and  $\tM_\Gamma$ 
is a cover over $M$ with deck group  $H_1(M)/\Gamma$.

Any \homeo\ $f:M\raw M$ lifts to the universal
Abelian cover. For smaller Abelian covers, $f$ lifts to
$\tM_\Gamma$ if and only if $f_*(\Gamma)=\Gamma$,
where $f_*:H_1(M)\raw  H_1(M)$ is the induced action.
For a cover $\tM_\Gamma$ to which $f$ does lift,
the induced action of $f_*$ on the deck group 
$H_1(M)/\Gamma$ is denoted $f_\Gamma$. If $\tf$ is
a lift of $f$ to $\tM_\Gamma$, a fundamental relation is
\begin{equation}\label{equilift}
\tf\circ\delta_{g} = \delta_{f_\Gamma(g)}\circ \tf,
\end{equation}
where $\delta_g$ is the deck transformation corresponding
to $g\in H_1(M)/\Gamma$. Thus $\tf$ commutes with all deck
transformations precisely when $f$ acts trivially on 
$H_1(M)/\Gamma$.

\subsection{Subgroups of finitely generated Abelian groups}\label{subgroup}
 If  $\Gamma\subset \Z^d$ is a rank
$k$ subgroup, then  there is a basis $\{u_1, \dots, u_d\}$ of
$\Z^d$ and positive integers 
$a_1, \dots, a_k$, so that
\begin{equation}\label{fundfact}
\{a_1 u_1, \dots, a_k u_k\}
\end{equation}
 is a basis for $\Gamma$. This
is usually given as a simple consequence of the Smith normal
form (for example, \cite{intmat}).
This fundamental fact implies that $\Gamma$ is \de{co-finite},
\ie\ the quotient group $\Z^d/\Gamma$ is finite, if and
only if $\Gamma$ has rank $d$. In this case it follows
that if $K$ is the order of $\Z^d/\Gamma$, then
$K \Z^d \subset \Gamma$ and further,  if we form
a matrix $M$ using a basis for the $\Gamma$ as the
columns, then the order of $\Z^d/\Gamma$ is 
$|\det(M)|$. 

For a subset $X$ of an Abelian group $G$, let
$\gen{X}$ be the subgroup
generated by $X$, and 
the positive semigroup generated by $X$ is 
  \begin{equation*}
\genp{X} = \{ n x : n\in \N, n>0, \ \text{and}\ x\in X\}
\end{equation*}
It is easy to see that
if $G$ is a finite Abelian group and
$X$ is a subset, then $\genp{X} = \gen{X}$.

The subgroup $\Gamma\subset \Z^d$ is called \de{pure} if
whenever $g\in\Gamma$ is divisible in $\Z^d$, it
is divisible in $\Gamma$, \ie\ if $g\in\Z^d$ and 
$mg\in\Gamma$ for some $m\not= 0$, then $g\in\Gamma$.
Thus $\Gamma$ is pure if and only if the integers
$a_j$ in \eqref{fundfact} are all equal to $1$, and
$\Gamma$ is pure if
and only if the quotient $\Z^d/\Gamma$ is torsion-free or trivial. 
Also, the subgroup $\gen{g}$ generated by a single element $g$
of $\Z^d$ is pure if and only if $g$ is indivisible in
$\Z^d$. In addition, if $\Gamma$ is pure in $\Z^d$, then
it is always a summand, \ie\ there is a subgroup $H\subset
\Z^d$ with $\Z^d = \Gamma \osum H$ (this and all direct
sums in this paper are internal). Equivalently, if
$\Gamma$ is pure, then any basis of $\Gamma$ can be extended
to a basis of $\Z^d$. Finally, since any subgroup $H$
of $\Z^d$ is free and thus isomorphic to some $\Z^k$,
if $\Gamma\subset H$ and $\Gamma$ is pure in $H$ ($g\in H$
with $mg\in\Gamma$ with $m\not = 0$ implies $g\in\Gamma$),
then any basis of $\Gamma$ can be extended to a basis
of $H$.
For a subset $X$ of an Abelian group $G$, let $P(X)$ denote the
smallest pure subgroup of $G$ which contains $X$. The 
group $P(X)$ is commonly called the \de{purification} of $X$.

\subsection{The Fried cover}\label{sect:friedcover}
Fried pointed out in \cite{friedtwist, friedcohom} that
 for many dynamical applications it is best to work
with covering spaces on which all lifts of $f$ commute
with all deck transformations. Such a cover is of
necessity Abelian (Lemma 1 in \cite{friedtwist}).
 From \eqref{equilift} it follows
that for such a cover $f$ must act like the identity on
the deck group and that  the largest such cover corresponds to
the subgroup $F' := im(f_* - \id)\subset H_1(M)$. 

The deck group of  this cover,  the
quotient $H_1(M)/F'$, will frequently have torsion.
For most of our applications we will work only with
torsion-free part. Letting  
$F$ be the purification of $F'$, $F := P(F')$, we see
that the largest cover with
free deck group on which all lifts commute with deck has
deck group $H_1(M)/F$. 
It is easy to check that $f_*(F) = F$ and
so $f$ always lifts to this cover.
This leads to
\begin{definition}\label{friedcover}
Given a compact surface $M$ and \homeo\ $f:M\raw M$,
the Fried cover, $\tM_F$,  of $(f, M)$ is 
covering space corresponding to the subgroup of $H_1(M)$ given by
$F = P(\im(f_*-\id))$. 
\end{definition}

\begin{remark}\label{friednontriv}
  Note that $\Z^d/\im(f_* - \id)$ is finite
if and only if $\det(f_*-\id)\not= 0$ using the first paragraph of
\S\ref{subgroup}.
This determinant is not zero
if and only if $1$ is not an eigenvalue of $f_*$. Thus
the Fried cover deck group $\Z^d/P(\im(f_* - \id))$ is
nontrivial if and only if $1\in\spec(f_*)$. When $\Z^d/P(\im(f_* - \id))$
is trivial, by convention, the Fried cover is $M$ itself.
\end{remark}

\subsection{Rotation sets}\label{coverrot}
In this section we recall the generalized rotation vector
of orbits of a \homeo\ $f$. This
notion has its origins in Schwartzman's asymptotic cycles
and is now a common tool (see, for example, \S11 in  \cite{bdamster}
and the introduction of \cite{jenkinson}).
The definition requires a covering space which has
 torsion-free deck group and on which all lifts of $f$ 
commute with all deck transformation. Thus it is natural
to use the largest such cover, namely the Fried cover given in
Definition~\ref{friedcover}.
Assume that $1$ is 
an eigenvalue of $f_*$ and so the
deck group, $H_1(M)/F$,  of the Fried Cover
$\tM$ is nontrivial, \ie\ $H_1(M)/F =\Z^d$
for $d>0$.

The definition of the rotation vector requires a means of
measuring displacements in the Fried cover which is 
compatible with the deck action. 
If $H_1(M) = \Z^k$,  a standard construction yields a 
continuous map on the universal Abelian cover
 $\tbeta':\tM \raw \reals^k$ with 
$\tbeta' \circ\delta_{\vn}(\tx) = \tbeta'(\tx) + \vn$
for all $\vn\in\Z^k$. Here is one way to do the construction.
Pick a set of generators and 
a basis $\vn_1, \dots, \vn_k$ for $H_1(M)$ 
and a corresponding co-basis $c_1, \dots, c_k$ for 
$H^1(M;\Z)$, \ie\ $c_i(\vn_j) = 
\delta_{ij}$ (Kronecker delta).
Now  treat $c_i$ as an element of  $H^1_{DR}(M,\R)$ and assume
that it is represented by the closed
one-form $\omega_i$.
Lift  $\omega_i$ to $\tomega_i$  on $\tM$, and fix
a basepoint $\tz_0\in\tM$. The $i^{th}$ coordinate
of $\tbeta'$ is $\tbeta'_i(\tz) = \int_{\tgamma} \tomega_i$, where
$\tgamma$ is any smooth arc in $\tM$ from $\tz_0$ to
$\tz$. We then project $\tbeta'$ to obtain
an equivariant map $\tbeta:\tM_F \raw \R^d$.

Now for for a given lift 
$\tf_F$ of $f$ to the Fried cover $\tM_F$ and an $x\in M$,
pick a lift $\tx\in\tM_F$ of $x$  and $n\in\Z$ and let 
\begin{equation}\label{liftcocycle}
B(x,n) = \tbeta(\tf_F^n(\tx)) - \tbeta(\tx).
\end{equation}
Since $\tf_F$ commutes with the deck group of
$\tM_F$ and $\tbeta$ is equivariant, this definition is independent
of the choice of $\tx$, but it does depend on the choice of 
$\tf$. 
It is immediate that $B$ is an additive
dynamic cocycle over $(M, f)$. We let $\rot(x,\tf_F)$ 
be the element of $\R^d$, 
\begin{equation*}
\rot(x,\tf_F ) := \lim_{n\raw\infty} \frac{B(x, n)}{n}
\end{equation*}
when the limit exists, 
and define the rotation set in $\R^d$ as
\begin{equation*}
\rot(\tf_F) = \{ \rot(x,\tf_F) : x\in M\}.
\end{equation*}
It is immediate that 
\begin{equation}\label{liftpower}
\rot(\delta_{\vn}\circ \tf_F^q) = q\rot(\tf_F) + \vn,
\end{equation}
for all $q\in\Z$ and $\vn\in\Z^d$.
Since $B(x, 1)$ is continuous on $M$, it is bounded and
thus is in $L^1$ of any $f$-invariant probability measure.
This implies by the point-wise ergodic theorem that
the rotation vector exists almost everywhere with respect to
such measures.

Given a \homeo\ $f$ of $M$ and a lift $\tf$ to the universal
Abelian cover $\tM$, there is a unique lift $\tf_F$ of $f$
to the Fried cover which is  the projection of $\tf$. Define
\begin{equation}\label{friedrot}
\rot_F(\tf) = \rot(\tf_F).
\end{equation}
By convention if the Fried cover is trivial
we let $\rot_F(f)$ be the empty set.

\section{Twisted Skew Products}
Twisted skew products over a subshifts of finite type
provide a symbolic model for  the lifts of rel \pA\ maps
to Abelian covers. 
For the inductive arguments based on Lemma~\ref{MIL} 
we also need to consider
a countable state Markov chain as the  base shift.

\begin{definition}\label{skewdef}
A twisted skew product is constructed from:
\begin{compactenum}

\item a countable state Markov shift $(\Sigma, \sigma)$ called the base
shift,
\item a finitely generated Abelian group $G$ called the group component,
\item a function $h:\Sigma\raw G$ called the height function which is 
required to be constant on central cylinder sets of length two,
 and so $h(\us) = h(s_0,s_1)$,
\item and an isomorphism $\Psi:G\raw G$ called the twisting automorphism.
\end{compactenum}
The twisted skew product built from these ingredients
is  the  map $\tau:\Sigma\times G\raw \Sigma\times G$
given by
\begin{equation}\label{skeweq}
\tau(\us, g) = \left(\sigma(\us), \Psi(g) + h(\us)\right).
\end{equation}

In the special case when the twisting
isomorphism  $\Psi = \id$, the map $\tau$ is called
a skew product or, for emphasis, an untwisted
skew product. The  group component most 
commonly considered here is $G = \Z^d$ and in this
case the twisting
automorphism $\Psi$ is given by a twisting matrix $A\in\SL(d,\Z)$.
\end{definition}

\begin{remark} Condition (c) in the definition of a twisted
skew product was adopted so that the skew could be easily identified
with a countable Markov shift as in \S\ref{conn}.
In the somewhat more general case when $h$ depends on length $(\ell+1)$-blocks,
one can pass to the $\ell$-block presentation of $(\Sigma, \sigma)$ 
(see page 27 in \cite{kitchens}), 
and the corresponding height function will depend on length-two central
blocks.  Since $G$ is a discrete group, when $h$ is continuous
on a subshift of finite type, there
will always be an $\ell$ with $h$ constant on central length
$\ell$ cylinder sets.
\end{remark}

Skew products as 
just defined are also  called  
 $G$-extension of the countable state Markov shift $(\Sigma, \sigma)$.
 The height function $h$ is also commonly called the 
 \de{cocycle} since in  the untwisted case
it generates an additive cocycle over
$(\Sigma, \sigma)$ as in \S\ref{height} below. We will sometimes
call a skew product  a twisted or untwisted extension of the
 base shift as is appropriate.

The iterate of a twisted skew product is
itself a twisted skew product. Specifically, 
if $\tau$ is as in \eqref{skeweq}, then
\begin{equation}\label{skewiteq}
\tau^k(\us, g) = \left(\sigma^k(\us), \Psi^k(g) + \hh^{(k)}(\us)\right),
\end{equation}
where $\hh^{(k)}(\us) = \Psi^{k-1}(h(\us)) +  \Psi^{k-2}(h(\sigma(\us))) 
+ \dots + \Psi(h(\sigma^{k-2}(\us))) + h(\sigma^{k-1}(\us))$.
In particular,
the twisting automorphism of $\tau^k$ is
$\Psi^k$.

Note that $G$ acts  by addition on the second component of
$\Sigma\times G$. For a $g\in G$, we write
this action of $g$ as 
\begin{equation}\label{Taction}
T_g(\us, g') = (\us, g + g').
\end{equation}
A simple calculations shows that 
\begin{equation}\label{equi}
\tau\circ T_g = T_{\Psi(g)} \circ \tau.
\end{equation}
which is the obvious analog of \eqref{equilift}.
Thus $\tau$ commutes with the action of $G$
if and only if $\Psi = \id$, \ie\ when $\tau$ is an untwisted
skew product.

\subsection{Quotients, the finite lifting property
 and the Fried quotient}\label{quotientsect}
Many of the standard constructions for covering spaces
with deck group $G$ such as quotients and
subcovers have analogs for twisted skew products
with group component $G$.
Given a twisted skew product  $\tau$ as in
\eqref{skewdef}, if $\Gamma\subset G$ is a 
subgroup with $\Psi(\Gamma)= \Gamma$,
then $\tau$ descends to a twisted skew product
$\tau_\Gamma:\Sigma \times (G/\Gamma)\raw \Sigma \times (G/\Gamma)$,
defined by
 \begin{equation}\label{quotientdef}
\tau_\Gamma (\us, g + \Gamma) = 
\left(\sigma(\us), \Psi(g) + h(\us) + \Gamma\right).
\end{equation}
Further, the projection $\pi: G \raw G/\Gamma$ induces
a semiconjugacy, $\id \times \pi$, from $\tau$ to $\tau_\Gamma$. 
Thus, in particular, if $\tau$ is transitive, so
is any quotient $\tau_\Gamma$.

More generally,
if $\Gamma_1\subset \Gamma_2 \subset G$ are two $\Psi$-invariant
subgroups, then it is easy to check that
\begin{equation*}
\begin{CD}
0 @>>> \Gamma_2/\Gamma_1 @>>> G/\Gamma_1 @>>>G/\Gamma_2 @>>> 0\\
@. @VVV @VVV @VVV @.\\
0 @>>> \Gamma_2/\Gamma_1 @>>> G/\Gamma_1 @>>>G/\Gamma_2 @>>> 0%
\end{CD}
\end{equation*}
commutes, where the vertical maps are the natural ones
induced by $\Psi$. Thus the induced skew product on 
$\Sigma\times (G/\Gamma_1)$  is semiconjugate
to the one on $\Sigma\times (G/\Gamma_2)$, and further, the Noether isomorphism
$(G/\Gamma_1)/(\Gamma_2/\Gamma_1) \cong G/\Gamma_2$ induces
a conjugacy of the twisted skew products on 
$\Sigma\times (G/\Gamma_1)/(\Gamma_2/\Gamma_1)$ and
$\Sigma\times (G/\Gamma_2)$.

 Rel \pA\ \homeos\ have a  very useful property which
was pointed out and used by Fried (\cite{friedcomm}). These maps 
are transitive and their lift to any 
covering space with finite deck group
is also transitive. We now
define the analogous property for a twisted skew
product.
Restrict now to the case where group 
component of the skew product is a free Abelian
group $\Z^d$. 
\begin{definition}\label{ftpdef}
The skew product $\tau:\Sigma\times\Z^d\raw \Sigma\times\Z^d$
with twisting automorphism, $\Psi$, 
is said to have the finite transitivity property (ftp) if
for every $\Psi$-invariant, co-finite subgroup $\Gamma\subset \Z^d$ 
 the quotient map
$\tau_\Gamma:\Sigma\times (\Z^d/\Gamma) \raw \Sigma\times (\Z^d/\Gamma)$ 
 is transitive.
\end{definition}

Note that if $\tau$ has the ftp then so does any quotient.
This follows from  the conjugacy induced by the Noether
isomorphism  given above.
Note also that the base-shift
is itself a quotient of of $\tau$ and so when $\tau$ has the ftp,
its base shift is always transitive.

Using \eqref{equi} and \eqref{equilift}, 
untwisted skew products correspond
to covering spaces on which all lifts commute with the
deck group. As with covering spaces, it
is often useful to pass to a quotient on which
the projection is untwisted. 
The construction given here is the exact analog of that
for covering spaces given in  \S\ref{sect:friedcover}. 
Specifically, if $\tau$ is a twisted skew product with
group component $\Z^d$ and twisting isomorphism $\Psi$,
the largest untwisted quotient corresponds to the 
$\Psi$-invariant subgroup $F' = im(\Psi - \id)$, and
the largest untwisted quotient with torsion-free group component
corresponds to the purification $F = P( F')$.

\begin{definition}\label{friedskew}
If $\tau$ is a twisted skew product with
group component $\Z^d$ and twisting isomorphism $\Psi$,
the Fried quotient of $\tau$ is the quotient
$\tau_F$ for  $F = P( im(\Psi - \id))$.
\end{definition}

\subsection{Twisted skew products, 
countable state Markov shifts, and lifted transitions}\label{conn}
A twisted skew product can be identified with, or more precisely 
is conjugate to, a countable state Markov shift in a 
natural way. Given $\tau$ as in \eqref{skewdef}, assume that  
the states of the base shift $\Sigma$ are $ S = \{ 1, 2, 3, \dots \}$.
Define the states of a new countable Markov
shift as  $\hS = S \times \Z^d$, and the allowable one-step transitions
 for the new shift $\hS$ are 
$(a, \vn) \raw (b, \vm)$, where $a\raw b$ is allowable for $\Sigma$ and
$\vm = \Psi(\vn) + h(a, b)$.  Now let $\hat{\Sigma}\subset \hS^\Z$
be the collection of sequences from  $\hS^\Z$ all of whose transitions 
are allowable and  $\hat{\sigma}$ be the left shift on $\hat{\Sigma}$. The
conjugacy between $(\hat{\Sigma},\hat{\sigma})$ and
$(\Sigma\times\Z^d, \tau)$ is given by
$(\dots, (s_{-1}, \vn_{-1}), (s_{0}, \vn_{0}), 
(s_{1}, \vn_{1}), \dots ) \mapsto (\us, \vn_0)$.
We will often identify a skew product and its corresponding
countable state Markov shift with little  further mention.

Given an allowable one-step 
transition $a\raw b$ for $(\Sigma,\sigma)$,
the  \de{lifted or induced} transition for $\tau$ starting at $g\in G$
is \begin{equation}\label{lifted}
(a, g) \raw (b, \Psi(g) + h(a,b)).
\end{equation}
More generally, if $a\raw c$ is an allowable  $n$-step transition,
its lifted transition for $\tau$ is constructed by concatenating
the lifts of each one-step transition.

After identifying a skew product with a 
countable state Markov shift  the criterion
for transitivity given in \S\ref{Markovdef} can be used. Thus
$\tau$ is transitive if and only if for any states
$(a, \vn)$ and $(b, \vm)$, there is an allowable transition
$(a, \vn) \raw (b, \vm)$. Equivalently, $\tau$ is transitive
 if and only if for any
$(a, \vn)$ and $(b, \vm)$, there is 
$\us\in\Sigma$ and an $n>0$ so that
$s_0=a$ and $ s_n = b$ and
$\tau^n(\us, \vn) = (\sigma^n(\us), \vm)$.

\section{Untwisted Skew Products}
In this section we review some of the standard constructions 
associated with untwisted skew products.

\subsection{Lifting untwisted transitions and transitivity}
For untwisted skew products lifted transitions transform
nicely under the action of $G$ from \eqref{Taction}.
 In this case,
\begin{equation}\label{additive}
 (a_1, g_1) \raw (a_2, g_2)\ \text{implies}\ 
(a_1, g_1 + g) \raw (a_2, g_2 + g),
\end{equation}  for all $g\in G$.
This implies, in particular, that if 
$(\us, 0)$ is periodic point for $\tau$, then for all $g\in G$,
$(\us, g)$ is also. In addition,  
\begin{equation}\label{skewadd}
 (a_1, g_1) \raw (a_2, g_2)\ \text{and}\  (a_2, g_3) \raw (a_3, g_4)
\ \text{implies} \ 
(a_1, g_1) \raw (a_3, g_4 + g_2 - g_3)
\end{equation} 
 If $\tau$ is
untwisted, we thus have that
 $\tau$ is transitive if and only
if for any $j, k\in S$ and $\vm\in\Z^d$, 
there is an allowable transition
$(j, 0) \raw (k, \vm)$.

\subsection{The height cocycle and the displacement set}\label{height}
A  main tool in the study of untwisted skew products
is the cocycle giving the total height or displacement
of an orbit in the group component. This object has
many names and notations, including the Fr\"obenious element and
the total displacement. In the covering space context it gives
a coordinate for the Abelian Nielsen class of a periodic point
(see \cite{bdamster}) or
the twisted Lefschetz coefficient (see \cite{friedtwist}).

Given untwisted skew product $\tau$ with height function $h$, let
\begin{equation*}
h(\us, n) = h(\us) + h(\sigma(\us)) + \dots + h(\sigma^{n-1}(\us)),
\end{equation*}
and so for any $g\in G$,
\begin{equation*}
\tau^n(\us, g) =(\sigma^n(\us), g + h(\us, n)).
\end{equation*}
Note that we are ``overloading'' the symbol $h$. Other common notations
for $h(\us, n)$ include 
$h^n(\us)$ and $h^{(n)}(\us)$.  The notation
we choose emphasizes the valuable and frequently used fact that for
an untwisted skew product the height function $h$ induces 
an additive cocycle over the base shift $(\Sigma, \sigma)$: 
\begin{equation*}
h(\us, n + m) = h(\us, n) + h(\sigma^n(\us), m).
\end{equation*}

\begin{remark}\label{diff}
It is worth noting here an important difference between
twisted and untwisted skew products. 
If $\eta$ is a twisted skew product,
then one can define an additive $G$-valued cocycle for $\eta$ itself by
$E((\us, g), n) = \pi_2(\tau^n(\us, g)) - g$ where
 $\pi_2$ is projection onto the group factor. However, this
will only descend to a cocycle on the base shift $\Sigma$ when the 
value of $E$ is independent of the element $g\in G$ and by virtue of
\eqref{equi}, this only happens when $\Psi = \id$, \ie\ when
$\eta$ is untwisted.  This basic fact is the reason that
twisted skew products present additional difficulties over
untwisted ones.
\end{remark}

For untwisted skew products
a  special role is played by the value of the cocycle
on periodic points of the base shift.
Given an untwisted skew product $\tau$, 
define the \de{displacement set} $D(\tau)$ as 
\begin{equation}\label{Ddef}
D(\tau) = \{ h(p, n) : p\in\Fix(\sigma^n), n > 0\}.
\end{equation}
 For future
reference we note that if $p\in\Fix(\sigma^n)$, then
$h(\sigma^k(p), n) = h(p,n)$ for all $k$.

In the situation described in \S\ref{quotientsect}
where  $\Gamma_1\subset \Gamma_2$ are subgroups
of $\Z^d$, then the natural
map $\pi:\Z^d/\Gamma_1 \raw \Z^d/\Gamma_2$ yields a simple
relationship between the displacement 
sets of  $\tau_{\Gamma_1}$ and $\tau_{\Gamma_2}$. 
Specifically, if the height functions of 
$\tau_{\Gamma_1}$ and $\tau_{\Gamma_2}$ are
$h_1$ and $h_2$, respectively, then
for a periodic point $p\in\Sigma$, we have 
$h_2(p,n) = \pi(h_1(p,n))$. It then follows that
\begin{gather}\label{displaceeq}
\genp{D(\tau_{\Gamma_2})}  = \pi(\genp{D(\tau_{\Gamma_1})})\notag  \\
\gen{D(\tau_{\Gamma_2})} = \pi(\gen{D(\tau_{\Gamma_1})})
= \gen{D(\tau_{\Gamma_1})}/(\gen{D(\tau_{\Gamma_1})}\cap \Gamma_2). 
\end{gather}

\subsection{The rotation set}
Given an  additive cocycles it is natural to compute the
asymptotic average values. In the case of the
displacement cocycle this average is the analog of
the rotation vector of a lifted \homeo\ and so for the
sake of uniform terminology we use that name here.

We continue to restrict to the case of 
an untwisted skew product $\tau$ and in addition we
require that the group factor be torsion-free, $\Z^d$.
If the height cocycle for $\tau$ is $h(\us, n)$,
for $\us\in\Sigma$ define
its \de{rotation vector} as the vector in $\R^d$ given by 
\begin{equation}\label{rotdef}
\rot(\us) = \lim_{n\raw\infty} \frac{h(\us,n)}{n},
\end{equation}
if the limit exist. For any invariant probability
measure $\mu$ on $(\Sigma, \sigma)$ if the
height function is integrable with respect to $\mu$
($h\in L^1(\mu)$), then by the point-wise ergodic
theorem the limit in \eqref{rotdef}
exists almost everywhere with respect to $\mu$.
The collection of all rotation vectors for $\tau$ is
called  the \de{rotation set} and is denoted
 \begin{equation*}
\rot(\tau) = \{ \rot(\us) : \us \in \Sigma \}. 
\end{equation*}
Note that for a $p\in\Fix(\sigma^k)$, $\rot(p) = h(p,k)/k$,
and for all $q>0$ and $\vp\in\Z^d$,
\begin{equation}\label{rotpower}
\rot(T_{\vp}\circ\tau^q) =  q \rot(\tau) + \vp,
\end{equation}
where $T_{\vp}$ is the action of $\vp$ given in \eqref{Taction}.

As noted in Remark~\ref{diff},  the height function does not descend
to a cocycle on the base shift when a skew product in nontrivially
twisted. Thus in the twisted case there  generally
is not a usable notion of rotation
vector. However,  valuable information on the
twisted skew product can be
obtained using the rotation set 
of the largest torsion-free quotient
on which $\tau$ descends to an untwisted skew product.
In analogy with Definition~\ref{friedrot} we have:
\begin{definition}\label{friedrotskew}
If $\tau$ is a twisted skew product, let 
 $\rot_{F}(\tau) = \rot(\tau_{F})$, where $\tau_{F}$ is the
Fried quotient of $\tau$ defined in  \S\ref{friedskew}.
When $\Z^d/F$ is the trivial group we adopt the convention
that  $\rot_{F}(\tau) = \emptyset$.
\end{definition}

\section{Symbolic Models for Lifted Pseudo-Anosov Maps}

\subsection{Pseudo-Anosov maps}\label{pA}
We briefly review a few properties of rel \pA\ maps of
relevance here. For more information see  \cite{flp, cb}.
A \de{pseudo-Anosov} \homeo\ $\phi$ 
of a surface $M$ is characterized by
the existence of a pair $(\cF^u,\mu^u), ( \cF^s, \mu^s)$ of transverse,
$\phi$-invariant measured foliations one expanding and
the other contracting.
A \homeo\ $\phi$ of a compact surface $M^2$ is called
\de{\pA\ relative to the finite set $A$}  if it has all the usual properties of
\pA\ maps but in addition, its invariant foliations
 have  one-prong
singularities on the set $A$. For brevity of terminology,
if $\phi$ is \pA\ relative to some finite (or empty) set, it is 
called \de{rel \pA}.

A rel \pA\ map always has a Markov partition containing 
a finite number of rectangles $\{ R_1, \dots, R_k\}$.
By subdividing the partition if necessary we may
assume that for each $i, j$, the intersection
$R_i \cap \phi(R_j)$ has at most one component. 
The transition matrix $C$ is a $k\times k$ matrix
defined by $C_{ij} = 1$ if $R_i \cap \phi(R_j)\not=\emptyset$
and $C_{ij} = 0$, otherwise. Let $\Sigma$ denote the subshift
of finite type constructed from the matrix $C$; 
this subshift is  always transitive (irreducible) and
topologically mixing. There
is a semiconjugacy $\alpha$ from $(\Sigma, \sigma)$ to
$(M^2,\phi)$ which is bounded  to one, 
is bijective on dense, $G_\delta$ sets, and the 
image of a cylinder set in $\Sigma$ is a topological
disk in $M$, and thus the \pA\ map  is
also transitive and topologically mixing.

\subsection{Twisted skew products corresponding to a lifted \pA}\label{pAinfo}
We now describe the construction of
a twisted skew product which is a symbolic model
for the lift of the \pA\ map $\phi$ to an Abelian cover. 
The process is quite standard, but we will need some details
of the construction below. In addition, the case when
$\phi$ is not isotopic to the identity doesn't seem
to have been described in the literature. 

Assume now that $\phi$ is a rel \pA\ map of the compact surface
surface $M$ with universal Abelian cover $\tM$ with Markov
partition, transition matrix $C$, subshift of finite type $(\Sigma,\sigma)$,
and semiconjugacy $\alpha$ as in \S\ref{pA}
A skew product corresponding to the lift $\tphi$ will
have base shift $(\Sigma,\sigma)$, group factor 
$H_1(M)=\Z^d$, and twisting automorphism $\phi_*$.
The height function $h$ measures how much 
a lifted rectangle moves in the cover and is defined
as follows.

Fix a fundamental domain $\tM_0'$ for $\tM$ and one lift $\tR_j$ of
each rectangle $R_j$  such that $\tR_j\cap\tM_0 \not=\emptyset$
for all $j$ and $\cup \tR_j $ is a connected set. Let
$\tM_0 = \cup \tR_j$. Note that $\tM_0$ is also a fundamental
domain for $\tM$.  Now
since $C$ is a $(0,1)$ matrix, if $a\raw b$ is an allowable 
one-step transition
for $\Sigma$,  we have 
\begin{equation}\label{thdef}
\tphi(\tR_a) \cap \delta_{\vn} (\tR_b) \not=\emptyset
\end{equation}
for exactly one $\vn\in\Z^d$. Define $h:\Sigma\raw\Z^d$ 
as constant on a length two cylinder set $[a, b]_0$ by 
$h(\us) = \vn$ for all $\us\in [a, b]_0$,
 where $\vn\in\Z^d$ is the unique deck element
for which \eqref{thdef} holds. 

Using \eqref{equilift}, changing the chosen fundamental domain from $\tM_0$ to
some $\delta_{\vm} (\tM_0)$  will change the height function $h$ by 
a constant to 
$h' = (\phi_*-I)(m) + h$. In addition, a given rel \pA\ map
is modeled by many (closely related) 
subshifts of finite type. Thus there are actually 
many skew products which correspond to $\tphi$. 



\subsection{The semiconjugacy}\label{semiconstruct}
We now construct a semiconjugacy $\talpha$ from
$(\Sigma\times\Z^d, \tau)$ onto $( \tM, \tphi)$.
 In making the construction it is initially
easier to work with the countable state Markov shift described
in \S\ref{conn} which is naturally conjugate to $\tau$.
Recall that this shift $(\hSigma, \hsigma)$
 has states $\hS = \{1, 2, \dots, k\} \times \Z^d$ and
so sequences $\hus\in \hSigma$ are given as
 $\hus = \dots, (s_{-1}, \vn_{-1}), 
(s_{0}, \vn_{0}), (s_{1}, \vn_{1}), \dots$. Define
 $\talpha:\hSigma\raw\tM$ by
\begin{equation}\label{talphadef}
\talpha(\hus) = \bigcap_{j\in\Z} \phi^{-j} (\delta_{\vn_j}(R_{s_j})).
\end{equation}
The basic properties of the Markov partition for $\phi$
coupled with \eqref{equilift} ensure that the intersection
in \eqref{talphadef} consists of exactly one point, 
$\talpha$ is continuous, onto, bounded to one, and
$\talpha \circ \hsigma = \tphi\circ \talpha$.
Further, $\talpha$ is  
is bijective on dense, $G_\delta$ sets, and the 
image of a cylinder set in $\hSigma$ is a topological
disk in $\tM$. Since the cylinder sets form
a basis for the topology of $\tM$, 
$\talpha$ is transitive  or topologically mixing if
and only if $\tphi$ has these properties.

As in \S\ref{conn} we can identify the Markov model for
$\tau$ with $\tau$ and so we can also consider the
semiconjugacy from $(\Sigma\times\Z^d,\tau)$ to
$(\tM,\tphi)$. We will also denote this semiconjugacy
as $\talpha$. With this identification the construction
of $\talpha$ gives that for each $\vn\in\Z^d$, 
\begin{equation}\label{domainimage}
\talpha(\Sigma\times\{\vn\}) = \delta_{\vn}(\tM_0).
\end{equation}

\subsection{Dynamical correspondence of skew product and lifted
\pA}\label{quasisect}
Using \S\ref{semiconstruct} and \S\ref{coverrot} we now have maps
\begin{equation}\label{maps}
\begin{CD}
\Sigma\times\Z^d@>\talpha>>\tM@>\tbeta>>\R^d.
\end{CD}
\end{equation}
An understanding of the metric properties of these
maps is necessary to compare the dynamics of $\tau$ and
$\tphi$ on these unbounded spaces. 
Recall that a map between pseudo-metric spaces
$q:(X,d)\raw (X', d')$ is called
a \de{quasi-isometry} if there are numbers $a>1$ and $b>0$ with
\begin{equation*}
\frac{1}{a}\; d(x_1, x_2) - b \leq d'(q(x_1), q(x_2)) 
\leq a \; d(x_1, x_2) + b
\end{equation*}
for all $x_1, x_2\in X$. A quasi-isometry yields a correspondence
of the metric structure on ``large scales''.

Define a pseudo-metric $d_1$ on $\Sigma\times\Z^d$ using the
norm on the group component, 
$d_1(\ut, \ut') = \| \pi_2(\ut)-\pi_2(\ut') \|$.
For a topological metric $d$ on the surface $M$, let 
$\td$ be its lift to the universal Abelian cover. As a consequence
of \eqref{domainimage}, $\talpha:(\Sigma\times\Z^d, d_1)\raw
(\tM, \td)$ is a quasi-isometry. 
In addition, the map $\tbeta':\tM\raw\R^d$ from \S\ref{coverrot}
is also a quasi-isometry from $(\tM, \td)$ to $\R^d$ with the
standard metric. 
Thus all the natural equivariant ways of measuring large
scale displacements on $\Sigma\times\Z^d$ and  $\tM$ are
all comparable. 

We also need to compare the corresponding linear structures
on the group factor of $\Sigma\times\Z^d$ and on the
vector space $\R^d$. Specifically, the construction
of $\talpha$ and $\tbeta$ yields for all  
$(\us, \vn)\in \Sigma\times\Z^d$,
\begin{equation}\label{linearcompare}
\left| \tbeta\circ\talpha(\us, \vn) - \vn \right| \leq\sqrt{d}
\end{equation}
(the $\sqrt{d}$ is the diameter of  the unit cube in $\R^d$).

Passing to Fried quotients, if $H_1(M)/F = \Z^k$, then 
 the maps in \eqref{maps} descend to 
\begin{equation*}
\begin{CD}
\Sigma\times\Z^k@>\talpha_F>>\tM_F@>\tbeta_F>>\R^k,
\end{CD}
\end{equation*}
with $\talpha_F$ a semiconjugacy from  $\tau_F$ 
to  $\tphi_F$.
Both $\talpha_F$ and $\tbeta_F$ are 
quasi-isometries of the projected pseudo-metrics and
the analog of \eqref{linearcompare} for $\tbeta_F\circ\talpha_F$ holds
as well. 

The next proposition summarizes the connections between
the lift of a rel \pA\ map $\tphi$ to the universal Abelian
cover $\tM$ and a corresponding twisted
skew product $\tau$.

\begin{proposition}\label{summary}
Assume that $\phi$ is a rel \pA\ map on the compact surface
$M$,  $\tM$ is the universal Abelian cover of $M$,
$\tphi$ is a lift of $\phi$ to $\tM$, and $\tau$ is a twisted
skew product with base shift $(\Sigma, \sigma)$ 
that corresponds to $\tphi$.  

\begin{compactenum}
\item $(\tM,\tphi)$ is transitive (topologically mixing) if and only if a 
$\tau$ is transitive (topologically mixing).
\item  $\tau$ has the ftp.
\item $\rot_F(\tphi) = \rot_F(\tau)$ and 
these sets have dimension $k  = \rank(\Z^d/F)$.
\end{compactenum}
\end{proposition}

\begin{proof}
The proof of (a) was indicated in \S\ref{semiconstruct}.
To prove (b), note that if $\Gamma\subset H_1(M)$ is co-finite
 and $\phi_*(\Gamma) = \Gamma$,
then $\tphi$ on the universal Abelian cover descends
to a $\tphi_\Gamma$  on the quotient cover $\tM_\Gamma = \tM/\Gamma$.
The map $\tphi_\Gamma$ is a lift of $\phi$ to the compact
surface $\tM_\Gamma$ and thus is
 itself \pA\ and so $\tphi_\Gamma$ is transitive.
The semiconjugacy $\talpha$ from $(\Sigma\times\Z^d, \tau)$ 
to $(\tM,\tphi)$ descends to a semiconjugacy 
$\talpha_\Gamma$ from $(\Sigma\times(\Z^d/\Gamma), \tau_\Gamma)$ 
to $(\tM_\Gamma,\tphi_\Gamma)$ and so 
$\tau_\Gamma$ is transitive as remarked at the end of
\S\ref{pA}.

To prove (c) assume that the Fried quotient is nontrivial.
Let $h_F$ be the height function of the projection $\tau_F$ 
of $\tau$ to the Fried quotient and $B$ be the cocycle
defined by \eqref{liftcocycle} on the Fried cover of $\phi$ and
$\alpha:\Sigma\raw M$ be the semiconjugacy 
from the base shift to the map $\phi$ on the surface.
Now if $\alpha(\us) = x$, then by construction,
$\talpha(\us, 0) := \tx$ is a lift of $x$. Thus using
\eqref{liftcocycle} and the fact that $\talpha_F$ is a semiconjugacy,
\begin{equation*}
B(x, n) = \tbeta_F(\tphi_F^n(\talpha_F(\us, 0))) 
- \tbeta_F(\talpha_F(\us, 0))
= \tbeta_F\circ\talpha_F(\tau_F^n(\us, 0)) 
- \tbeta_F\circ\talpha_F(\us, 0)
\end{equation*}
Since by definition,
 $h_F(\us, n) = \pi_2(\tau_F^n(\us, 0)) -  \pi_2(\us, 0)$ and
so by the Fried quotient version of \eqref{linearcompare} 
we have that  for any $\us\in\Sigma$ and 
$n\in \N$,
\begin{equation}\label{cohom}
|h_F(\us, n) - B(\alpha(\us), n)| \leq 2 \sqrt{k}.
\end{equation}  

From \eqref{cohom} and the definitions
of the rotation vectors it then follows
directly that  for $\us\in\Sigma$,  $\rot(\us, \tau_F)$ exists
if and only if $\rot(\alpha(\us), \tphi_F)$ does, and if
they exist they are equal, and thus the first statement in (c) follows. 
The assertion regarding
the dimension will be proved in Remark~\ref{pAtopdim} below.
\QED\

\end{proof}


\section{Transitivity of Untwisted Skew Products}
The main goal of this section is to prove the
following  theorem which gives conditions for transitivity 
of an untwisted skew product using the
rotation set. It is ultimately based on the
criterion for transitivity using the displacement set
provided by Coudene's Theorem~\ref{coudene} below, however
the rotation set is more tractable under the iterations and translations of 
maps required here.

\begin{theorem}\label{rotthm}
Assume that $\tau$ is an untwisted skew product with 
group factor $G=\Z^d$ and
the base shift $(\Sigma, \sigma)$ a transitive subshift of 
finite type. The following are equivalent:
\begin{compactenum}
\item $\tau$ is transitive, 
\item $\tau$ has the ftp and $0\in Int(\rot(\tau))$, 
\item $\tau$ has the ftp and its periodic points are 
dense in $\Sigma\times\Z^d$.
\end{compactenum}
\end{theorem}

\subsection{Coudene's transitivity theorem}
The next theorem gives conditions on the displacement set which ensure a
untwisted skew product is transitive. 
While the theorem given here is
stated in more general terms than Theorem 9 in 
\cite{coudene},  the method of proof indicated there
works for the version given here. 
 We shall need the generalization to 
countable state Markov shifts in the main induction argument below.
In covering space language the analog of this theorem
 was proved in \cite{bgh} for the case of $G=\Z$ and in
 \cite{parwanithesis} for $G = \Z^2$.

\begin{theorem}[Coudene]\label{coudene}
 Let $\tau:\Sigma\times G \raw \Sigma\times G$ be an untwisted skew
product with base shift $\Sigma$ a transitive countable  Markov shift
 and group component $G$ a finitely
generated Abelian group. The untwisted skew product $\tau$ is transitive
if and only if $\genp{D(\tau)} = G$.
\end{theorem}

\begin{remark}
The analog of Theorem~\ref{coudene} for lifted measures
is quite different in character. Let 
$\tau$ be an untwisted skew product over the subshift of
finite type  $(\Sigma, \sigma)$ with group component
$\Z^d$ and height function $h$ and 
$\mu$ be an ergodic, shift-invariant Gibbs measure with
 $\int h \;d\mu = 0$. 
Rees and then Guivarc'h 
 showed that the lift of $\mu$ to a $\tau$-invariant
(infinite) measure is ergodic if and only if $d = 1$ or $2$
(\cite{rees, guiv}).
\end{remark}

\subsection{The finite lifting property and transitivity}
For untwisted skew products we first obtain various conditions
equivalent to the ftp 
using the displacement set $D(\tau)$. Note that 
for untwisted
skew products by definition the twisting isomorphism $\Psi = Id$ and
 so the quotient
$\tau_\Gamma$ is defined for all subgroups $\Gamma\subset G$.

\begin{lemma}\label{ftp1} 
Assume that $\tau$ is an untwisted skew product with 
group factor $G=\Z^d$ and
base shift $(\Sigma, \sigma)$ a transitive countable
Markov shift. The following are equivalent:
\begin{compactenum}
\item $\tau$ has the ftp,
\item $\gen{D(\tau)} = \Z^d$,
\item For every co-finite $\Gamma$, 
$ \genp{D(\tau_\Gamma)} = \gen{D(\tau_\Gamma)} = \Z^d/\Gamma$.
\end{compactenum}
\end{lemma}

\begin{proof}
Since the group factor of $\tau_\Gamma$ is
$\Z^d/\Gamma$, Theorem~\ref{coudene} says that $\tau_\Gamma$ is transitive
if and only if $\genp{D(\tau_\Gamma)} = \Z^d/\Gamma$.
If $\Gamma$ is co-finite, as remarked in \S\ref{subgroup}, 
$\genp{D(\tau_\Gamma)} = \gen{D(\tau_\Gamma)}$ 
in the finite group $\Z^d/\Gamma$. 
Thus we have the equivalence of (a) and (c).
Now assume (b), as noted in \eqref{displaceeq}, $\gen{D(\tau_\Gamma)} = 
\gen{D(\tau)}/ (\gen{D(\tau)}\cap\Gamma) = \Z^d/\Gamma$,
and so (c) follows.

Now we prove not (b) implies not (c). If (b) is false, then
$\hat{H} := \gen{D(\tau)}$ is a proper subgroup of $\Z^d$. 
It follows easily from the fundamental fact \eqref{fundfact} 
that there is always a co-finite $H$ with $\hat{H}\subset H$ 
and $H$ is \textit{not} pure, and thus $|\Z^d/H| > 1$. 
Again using 
\eqref{displaceeq} we have,
$\gen{D(\tau_H)} = \gen{D(\tau)}/ (\gen{D(\tau)}\cap H) = 
\hat{H}/\hat{H} \not\cong \Z^d/H$, finishing the proof. 
\QED\ 
\end{proof}

Using the ftp we also get new conditions for transitivity again
for untwisted skew products.
\begin{theorem} \label{ftp2}
Assume that $\tau$ is an untwisted skew product with 
group factor $G=\Z^d$ and
the base shift $(\Sigma, \sigma)$ a transitive countable
state Markov shift. The following are equivalent:
\begin{compactenum}
\item $\tau$ is transitive,
\item $D(\tau) = \gen{D(\tau)} = \genp{D(\tau)}=\Z^d$,
\item $\tau$ has the  ftp and  there is
a rank $d$-subgroup $H\subset\genp{D(\tau)}$,
\item There exists a finite
set $\{d_i\}\subset D(\tau)$ so  that $\gen{\{d_i\}} = \Z^d$
and $0 = \sum a_i d_i$ with the $a_i$ positive integers.
\end{compactenum}
\end{theorem}

\begin{proof}
If $\tau$ is transitive, then
 for every $\vn\in\Z^d$ and state $s_0$ for
$\Sigma$, treating $\tau$ as a countable state Markov
shift, there is an allowable transition 
$(s_0, 0) \raw (s_0, \vn)$. This implies
that $\sigma$ has a period point with displacement $\vn$.
Thus, (a) implies that $D(\tau) = \Z^d$, and since obviously
$D(\tau) \subset \genp{D(\tau)}\subset \gen{D(\tau)}\subset\Z^d$,
the other equalities in (b) follow. Conversely, Theorem~\ref{coudene}
shows that (b) implies (a).

Now assume (d). By hypothesis $\gen{\{d_i\}} = \Z^d$ and thus
for any $\vn\in\Z^d$ there are integers $b_i$ with
$\vn = \sum b_i d_i$. Since the given
$a_i > 0$ in (d), we may find   $k>0$ with $k a_i + b_i > 0$ for
all $i$. Thus $\vn = \sum (k a_i + b_i) d_i$ and so
$\Z^d = \genp{\{d_i\}}\subset \genp{D(\tau)}$, implying
that $\Z^d =\genp{D(\tau)}$, and so $\tau$ is transitive 
by Theorem~\ref{coudene}, showing that (d) implies (a).

Now assume (b) and so $D(\tau) = \Z^d$. Let
 $d_1, \dots, d_{d}$ be a basis
for $\Z^d$ and for $ i = 1, \dots, d$, let $d_{d + i} = - d_{i}$.
Thus $\gen{\{d_i\}} = \Z^d$ and $\sum d_i = 0$ as required for (d).

Now assume (c). As noted in \eqref{displaceeq}, if $\pi_H: \Z^d\raw \Z^d/H$ is
the projection, then $\pi_H(\genp{D(\tau)}) = \genp{D(\tau_H)}$ and 
$\pi_H(\gen{D(\tau)}) = \gen{D(\tau_H)}$. Since $H$ has rank $d$,
it is co-finite and so
$\gen{D(\tau_H)} = \genp{D(\tau_H)}$, thus  
$ \pi_H(\gen{D(\tau)})= \pi_H(\genp{D(\tau)})$. Now since we
are assuming that $\tau$ has the ftp, by Lemma~\ref{ftp1} we have
$\gen{D(\tau)} = \Z^d$, and so  $\pi_H(\Z^d) = \pi_H(\genp{D(\tau)})$.
Thus, for each $\vn\in\Z^d$ there is a $\vm\in \genp{D(\tau)}$ with
$\vn + H = \vm + H$.  This means that some $h\in H$, 
$\vn = \vm + h$. Thus since $h\in H\subset  \genp{D(\tau)}$, we have
that $\vn\in  \genp{D(\tau)}$.  Since $\vn$ was
arbitrary, $ \genp{D(\tau)} = \Z^d$ and so
by Theorem~\ref{coudene} again, $\tau$ is transitive. Thus (c)
implies (a).

As remarked in \S\ref{quotientsect}, for any subgroup
 $\Gamma\subset\Z^d$, $\tau$ is semiconjugate to $\tau_\Gamma$,
and so if $\tau$ is transitive, then so is $\tau_\Gamma$. Thus
(a) coupled with (b) implies (c) using $\Z^d$ itself
as the rank $d$-subgroup $H$,
 finishing the proof.
\QED\
\end{proof}

\subsection{Rotation sets over subshifts of finite type}\label{rotsub}	
A few of the fundamental properties of the rotation set over
a transitive subshift of finite type will be needed in the sequel.
Basic results were obtained by Fried in \cite{friedtop} and
\cite{friedcomm} in the context of suspension flows, and much more
detailed  results were obtained by Ziemian in 
\cite{ziemian}  for the more general case of the Birhoff 
averages of any bounded vector-valued
function (see also  \cite{symbolpolytope}).

Recall that for a subshift of finite type $(\Sigma,\sigma)$,
a \de{simple block} is an allowable block that starts
and ends with the same symbol and contains no other symbol more
than once. Simple blocks are also sometimes called elementary blocks
or minimal loops. A \de{simple periodic point} is a periodic
point of $\sigma$ constructed
by infinite concatenation of a simple block. 
Let
\begin{equation*}
D_{simp}(\tau) = \{ h(p, n) : p\ \text{is a simple periodic point
with period}\ n\}.
\end{equation*}
It follows immediately from the definition of simple blocks
that any sequence in  $\Sigma$ can be written as the 
concatenation of such blocks. In particular,
 a periodic point $p$ of period $n$ can be constructed by concatenating
simple blocks $p_1, p_2, \dots, p_k$ of periods $n_1, n_2, \dots, n_k$
with $n = \sum p_j$. Thus 
$h(p, n) = \sum h(p_j, n_j)$ which implies that $\genp{D(\tau)}\subset
\genp{D_{simp}(\tau)}$. Since the other implication is trivial,
 \begin{equation}\label{simpledis}
\genp{D(\tau)} = \genp{D_{simp}(\tau)}\ \text{and similarly,} 
\ \gen{D(\tau)} = \gen{D_{simp}(\tau)}.
\end{equation}
For a linear transformation $L:\Z^d\raw \Z$, let
$\hL$ denote its linear extension to
$\hat{L}:\R^d\raw\R$.

\begin{theorem}\label{friedthm}
Assume that  $\tau$ is an untwisted skew product
 with base shift a transitive subshift of finite type and
group factor $\Z^d$.
\begin{compactenum}
\item  The rotation set
$\rot(\tau)$ is equal to the convex hull of the rotation vectors
of the simple periodic points.
\item Assuming that $0\in\rot(\tau)$,
 the rotation set $\rot(\tau)$ is $d$-dimensional
if and only if $\gen{D(\tau)}$ is a rank $d$-subgroup of $\Z^d$
\item If $\tau$ has the ftp and $0\in\rot(\tau)$,
 then rotation set $\rot(\tau)$ is $d$-dimensional.
\end{compactenum}
\end{theorem}

\begin{proof} For a proof of (a), see Theorem 3.4 in \cite{ziemian},
\cf\  Lemma 3 in \cite{friedtop} and Proposition 3.2 in  
\cite{symbolpolytope}.

We now prove the equivalence in (b). From (a) we know that
$\rot(\tau)$ is a convex hull in $\R^d$ with extreme points in
$\Q^d$, and by assumption,  $0\in\rot(\tau)$. This implies that
$\dim(\rot(\tau)) < d$ if and only if $\rot(\tau)$ is in the kernel of some
linear $\hat{L}$ which is the extension of a linear,
onto $L:\Z^d\raw\Z$. 
Now for each simple
periodic point $p_i$ of period $n_i$, $n_i\rot(p_i) = h(p_i,n_i)$,
and from \eqref{simpledis}, $\gen{D(\tau)} = \gen{D_{simp}(\tau)}$. 
This implies that $\rot(\tau)\subset\ker(\hat{L})$ if and only if 
$\gen{D(\tau)}\subset \ker(L)$. But $\gen{D(\tau)}\subset \ker(L)$
for some linear, onto   $L:\Z^d\raw\Z$ if and only if
$\gen{D(\tau)}$ has rank less than $d$. This proves (b).

If $\tau$ has the ftp, by Lemma~\ref{ftp1}  $\gen{D(\tau)}=\Z^d$,
and so (c) follows from (b). 
\QED\
\end{proof}

\begin{remark}\label{pAtopdim} 
We now give the proof of  the assertion regarding
the dimension in Theorem~\ref{summary}(c). 
Assume that $\tphi$ and $\tau$ correspond
and $\Z^d/F = \Z^k$ for $k>0$ (if $k=0$ there is
nothing to prove). In addition, let $\talpha_F$ be
the projection of the semiconjugacy between 
 $\tphi$ and $\tau$ to one between 
 $\tphi_F$ and $\tau_F$

If $\us$ is a 
periodic point for the base shift, then 
$\rot(\us, \tau_F) = \vp/q\in\Q^k$, and
so  if $\tx_F = \talpha_F(\us,0)$, then  
 $\rot(\tx_F, \tphi_F) = \vp/q$ also, and 
thus by \eqref{rotpower},
$0\in\rot(\delta_{-\vp}\; \tphi_F^q) = \rot(T_{-\vp}\; (\tau_F)^q)$. 
Now $\delta_{-\vp}\; \tphi_F^q$ is a lift to
$\tM_F$ of the \pA\ map $\phi^q$, and so
 $T_{-\vp}\; (\tau_F)^q$ has the ftp.  Thus by 
Theorem~\ref{friedthm}(c), $\rot(T_{-\vp}\; (\tau_F)^q)$
has dimension $k$ and so using \eqref{liftpower}, 
so does $\rot(\tau_F) = \rot(\tphi_F)$. 
\end{remark}

The next lemma describes another  property of the rotation
set over a subshift of finite type):  if
there is a large enough displacement in any ``direction''
in the cover, then there is a point with rotation
vector in that direction. One consequence is
a useful condition for deciding if zero is in the interior of a rotation
set. The author learned this argument from David Fried in 1983, and
a version of it was used in \cite{bgh}.

\begin{lemma}\label{boundlem}
Assume that $\tau$ is a untwisted skew product with group
component $\Z^d$,   height function $h$, and base shift
the transitive subshift of finite type $(\Sigma,\sigma)$.
For any onto linear functional $L:\Z^d\raw\Z$, there exists
a $C>0$ so that if there is a point $\us\in\Sigma$ and 
a positive integer $n$ with $L(h(\us,n)) > C$,  
then there is a period-$n'$ point $\us'$ with 
$\hL(\rot(\us'))> 0$, and if  there are $\us$ and
$n$ with $L(h(\us,n)) <- C$, then there is a $\us''$ with 
$\hL(\rot(\us''))<0$.

Thus, $0\in Int(\rot(\tau))$ if and only if for 
every linear, onto $L:\Z^d\raw\Z$, 
\begin{equation}\label{nobound}
\sup \{ L( h(\us,m)): \us\in\Sigma, m\in\N \} = \infty.
\end{equation}
\end{lemma}

\begin{proof}
Since $(\Sigma,\sigma)$ is transitive, for each pair of symbols
$i,j$ there is an allowable transition $i\raw j$ for $\sigma$.
Define  $d_{i,j}\in\Z^d$ as the group coordinate of the
lift of these transitions to transitions for $\tau$, so
$(i,0) \raw(j, d_{i,j})$. Given $L$, 
let $C_1 = \max_{i,j}\{ |L(d_{i,j})| \}$ and
$C = \max\{2 C_1, 1 \}$.  Assume now that for some 
$\us\in\Sigma$ and $n>0$ we have $L(h(\us,n)) > C$. 
The first $n+1$-symbols $s_0 s_1 \dots s_n$ in the given sequence $\us$ give
 $s_0\raw s_n$.  Using the transitivity of  $(\Sigma,\sigma)$, 
we have $s_n\raw s_0$. The concatenation of these two allowable
transitions is 
$s_0\raw s_n \raw s_0$ which lifts to
\begin{equation*}
(s_0, 0) \raw (s_n, h(\us, n)) \raw (s_0, h(\us,n) + d_{s_n,s_0}).
\end{equation*}
Thus if $\us'$ is constructed by infinite concatenation of
the block corresponding to  $s_0\raw s_n \raw s_0$ and
$n'$ is the sum of $n$ and the length of the transition
 $s_n\raw s_0$, then  $\us'$ is a periodic point
of period $n'$ and  $L(h(\us',n')) = L(h(\us,n)) + L(d_{s_n,s_0}) > C/2$.
 Since $\rot(\us',n') = h(\us',n')/n'$ and $\hL$ is linear
we have $\hL(\rot(\us'))> C/(2n')$. The construction of $s''$ is similar.

To prove the last statement of the lemma, first
note that $0\in Int(\rot(\tau))$ if
and only if for every linear, onto $L:\Z^d\raw \Z$ there exists
an $r\in\rot(\tau)$ with $\hL(r) > 0$. Assume now that 
for every linear, onto $L:\Z^d\raw\Z$, \eqref{nobound} holds.
Thus, in particular, if $C(L)$ is the constant depending
on $L$ given by the first paragraph of the theorem, there is
a point $\us$ and an $m>0$ with $L(h(\us,m))>C(L)$ and so
 there is a point
$\us'$ with $\hL(\rot(\us'))>0$, and so $0\in Int(\rot(\tau))$. Now conversely,
assume that there exists a linear, onto $L$ such that for all 
$\us$, there is a constant $K$ with $\sup_{m\in\N} L( h(\us,m))< K$.
Then certainly for any $\us$ for which  the rotation vector
exists, we have $\hL(\rot(\us)) \leq 0$, and so   $0\not\in Int(\rot(\tau))$.
\QED\
\end{proof}

\begin{remark} Note that the lemma does not say that there
always exists a point $\us$ as in the first paragraph
of the statement. A trivial example is when
$\Sigma$ is the full two-shift and $h\equiv 0$.
\end{remark}

\subsection{Proof of Theorem~\ref{rotthm}}\label{oldproof}
Before giving the proof of the theorem stated at the
beginning of this section we need a small fact about convex polytopes.
If $\{x_1, \dots, x_k \}$ is a finite set of points in $\Q^d$
and zero is in the interior of their convex hull,
then are positive integers $b_j$ with 
$0 = \sum_{i=1}^k b_j x_j$. The proof is an exercise.

\medskip
\textbf{Proof of Theorem~\ref{rotthm}:}
If $\tau$ is transitive, then certainly any quotient
is transitive and so $\tau$ has the ftp. To show that 
(a) implies (c) we must show that transitivity of 
a countable Markov shift implies it has dense periodic
points. This is standard: given any allowable block
$s_0 \dots s_n$, we must find a periodic point which
contains that block. Since the shift is transitive,
there is an allowable transition $s_n\raw s_0$. The
concatenated transitions $s_0 \raw s_n\raw s_0$ repeated
infinitely often yields the required periodic point. 
If  $0\not\in Int(\rot(\tau))$, then it follows from
the second statement of Lemma~\ref{boundlem} 
that $\tau$ is not transitive, and so (a) implies (b).

Now assume (c). Since $\tau$ has a periodic point,
certainly $0\in\rot(\tau)$ and since $\tau$ has the ftp,
by Theorem~\ref{friedthm}(c), $\rot(\tau)$ is $d$-dimensional. Thus if 
$0\in Fr(\rot(\tau))$ there is a linear, onto
$L:\Z^d\raw\Z$ with $\hL(r) \geq 0$ for all $r\in \rot(\tau)$ 
and using Theorem~\ref{friedthm}(a), there is a  point
$\us\in\Sigma$ with $\hL(\rot(\us)) > 0$.
Thus, we may find
an $m>0$ with $L(h(\us,m)) > 2 C$, where $C$ is the constant
associated with $L$ from Lemma~\ref{boundlem}. Now, since by
hypothesis periodic points
of $\tau$ are dense, treating $\tau$ as a Markov shift we
see that there must be a periodic orbit $t'$ of
$\tau$ that begins with the block
 $(s_0, 0),  (s_1, \vn_1),  \dots (s_m, \vn_m)$ with 
 $n_j = \vn_{j-1} + h(s_{j-1}, s_j)$ and $L(\vn_M) > 2C$.
 Since $t'$ is periodic, it must continue with a block
$(s_m, \vn_m),  (s_{m+1}', \vn_{m+1}),  \dots (s_{m+k}', \vn_{m+k}), (s_0, 0)$
for some $k>0$. Thus if $\hat{\us}$ is any allowable sequence for
$\sigma$ beginning with $s_m,  s_{m+1}',  \dots s_{m+k}', s_0$,
we have that $L(h(\hat{\us}, k + 1)) < -C$. Thus by Lemma~\ref{boundlem},
there is a  point $\us'$ with  $\hat{L}(\rot{\us'}) < 0$,
a contradiction.

Now  assume (b). Let $\{p_i\}$ be the finite set of simple
periodic points of $\sigma$  and let $d_j = h(p_j, k_j)$
where $k_j$ is the period of $p_j$. Lemma~\ref{ftp1} coupled
with the fact that $\gen{D_{simp}(\tau)} = \gen{D(\tau)}$ 
yields that $\gen{\{d_j\}} = \Z^d$. 
Since each $\rot(p_j) = d_j/k_j\in\Q^d$,
by the small fact on convex polytopes given above,
 $0 = \sum b_j\rot(p_j)$ with all the $b_j$ positive integers.  
 Thus letting  $a_j = b_j (\prod k_i) /k_j$, 
 we have
$0 = \sum a_j d_j$ with all $a_j$ positive integers, and so by
Theorem~\ref{ftp2}, $\tau$ is transitive.
 \QED\

\section{Examples}\label{examples}
We collect some simple examples which illustrate
 why various hypotheses are required.
 In all these examples the base shift
is the full two-shift and the skew product is untwisted.
The simple blocks of  the full two-shift written
as one-step transitions are $1\raw 1$, 
$0\raw 0$, $0\raw 1\raw 0$, and $1\raw 0\raw 1$.
We compute $\rot(\tau)$, 
$\gen{D(\tau)}$ and
 $\genp{D(\tau)}$ using the simple periodic points as per 
equation \eqref{simpledis} and Theorem~\ref{friedthm}(a) .
\begin{compactenum}
\item  For $\eta_1$, let the group component be $\Z$ and the
height function  $f_1\equiv 1$. Then 
$\rot(\eta_1) = \{1\}$,
$\gen{D(\eta_1)} =\Z$ and $ \genp{D(\eta_1)} =\{1, 2, \dots\}$.
So $\eta_2$ has the ftp, is not transitive and its rotation
set lacks interior, and has no periodic points.

\item For $\eta_2$, let the group component be $\Z$ and the
height function $f_2$ be given by $f_2(00) = -2$, $f_2(11) = 2$,
 $f_2(01) = f_2(10) = 0$. Then $\rot(\eta_2) = [-2,2]$ and  
$\gen{D(\eta_2)} =\genp{D(\eta_2)} = 2\Z$, and so $\eta_2$ is
not transitive and does not have the ftp. Thus $0\in Int(\rot(\eta_2))$
does not suffice to imply transitivity. Note also that
under $\eta_2$, $\Sigma\times\Z$ splits into a pair of transitive
subsystems,  $\Sigma\times(2\Z)$ and  $\Sigma\times(2\Z +1)$.
Thus, in particular, the periodic points of $\eta_2$ are
dense, but $\eta_2$ is not transitive.

\item For $\eta_3$, let the group component be $\Z$ and the
height function $f_3$ be given by $f_3(00) = -1$, $f_3(11) = 1$,
 $f_3(01) = f_3(10) = 0$. Then $\rot(\eta_3) = [-1,1]$ and  
$\gen{D(\eta_3)} =\genp{D(\eta_3)} = \Z$, and so $\eta_3$ 
is transitive. But note that $\eta_3^2 = \eta_2$ from
the previous example. Thus $\eta_3$ is transitive, but 
$\eta_3^2$ is not.

\item  For $\eta_4$, let the group component be $\Z^2$ and the
height function $f_4$ be given by $f_4(00) = (1,0) = f_4(01)$
and  $f_4(11) = (0,1)= f_4(10)$.
 Then $\rot(\eta_4)$ is the line
segment connecting $(1,0)$ and $(0,1)$ and has no interior.
Also, $\gen{D(\eta_4)} =\Z^2$ and so $\eta_4$ has
the ftp. Finally, $\genp{D(\eta_4)} =\{ (a,b) : a > 0, b>0\} - \{(0,0)\}$,
and so $\eta_4$ is not transitive. This illustrates that
the ftp alone doesn't imply that the rotation set has
interior.
\end{compactenum}

\section{Lemmas}
\subsection{Subgroup lemma}
The following elementary, technical lemma will be essential in the proof
of the main induction step given in Lemma~\ref{MIL}.
\begin{lemma}\label{GTL}
If $\Gamma\subset\Z^d$ is a rank $d$-subgroup and
$w\in\Z^d$ is a given element, then there exist
elements $g_1, g_2, \dots, g_{2d}\in\Gamma$ and a
 rank $d$-subgroup $H$ so that 
\begin{equation*}
H\subset \genp{g_1 + w, g_2 + w, \dots, g_{2d} + w}
\end{equation*}
\end{lemma}

\begin{proof}
To begin assume that we are in the special case where
the given $w$ is actually an element of $\Gamma$.
Let $\hatg_1, \dots, \hatg_d$ be a basis for $\Gamma$, and
for $1\leq j \leq d$, let $g_j = \hatg_j - w$, and
for $d < j \leq 2d$,  let $g_j = -\hatg_j - w$. Then
clearly $\genp{g_1 + w, g_2 + w, \dots, g_{2d} + w} =
\Gamma$,  so in this special
case the required $H$ is just $\Gamma$ itself. Henceforth
we assume that $w\not\in\Gamma$.

The strategy of the rest of the proof is to show that 
we may choose the elements $g_1, g_2, \dots, g_{2d}\in\Gamma$
so that if $k_j = g_j + w$, then there are positive integers
$c_j$ with 
\begin{equation*}
\sum_{j=1}^{2d} c_j k_j = 0.
\end{equation*}
Thus for any $j_0$, 
\begin{equation*}
-c_{j_0} k_{j_0} = \sum_{j\not=j_0} c_j k_j, 
\end{equation*}
and so, $-c_{j_0} k_{j_0}\in \genp{\{k_j\}}$. 
Now obviously,  $c_{j_0} k_{j_0}\in \genp{\{k_j\}}$,
and so 
\begin{equation*}
H := \gen{c_1 k_1, \dots, c_d k_d} \subset \genp{\{k_j\}}.
\end{equation*}
The last step in the proof is to show that the elements
  $\{ c_1k_1, \dots,  c_dk_d\} $ are linearly independent, yielding
that $H$ has rank $d$. 

To implement the strategy, as noted in \S\ref{subgroup},
if $q = |\Z^d/\Gamma|$, then $q w\in\Gamma$. Let $\barg_1$ be
primitive in $\Gamma$ and in the same direction
as $w$, i.e. for some $p>0$, $p \barg_1 = q w$. Now also from \S\ref{subgroup},
$\gen{\barg_1}$ is a pure subgroup of $\Gamma$ and so 
$\barg_1$ is an element of a basis $\{\barg_1, g_2,  \dots, g_d\}$ of $\Gamma$.

Now let $n_0\in\Z$ be such that 
\begin{equation}\label{ineq}
n_0 < -p/q < n_0 + 1.
\end{equation}
Note that we have strict inequalities in \eqref{ineq} because
$w\not\in\Gamma$. Now define elements of $\Z^d$,
\begin{eqnarray*}
k_1 &= (n_0 + 1 + p/q)\barg_1 &= (n_0 + 1) \barg_1 + w \\
k_{d+1} &= (n_0  + p/q)\barg_1 &= n_0 \barg_1 + w 
\end{eqnarray*}
and integers
\begin{eqnarray*}
A &= q n_0 + q + p &= q(n_0 + 1 + p/q)\\
B &= -(q n_0  + p) &= -q(n_0 + p/q).
\end{eqnarray*}
Note that $A>0, B>0$ and a simple calculation shows
that 
\begin{equation}\label{part1}
A k_{d+1} + B k_1 = 0,
\end{equation}
and so
$\gen{B k_1} \subset \genp{k_1, k_{d+1}}$. Thus
letting $g_1 = (n_0 + 1)\barg_1$ and
$g_{d+1} = n_0\barg_1$, this
completes the first step of the strategy when $d=1$,
so now assume $d>1$.

Let
\begin{eqnarray*}
k_j &=& g_j + w \ \text{for}\ 2\leq j \leq d\\
k_j &=& -g_j + w \ \text{for}\ 2+d\leq j \leq 2d.
\end{eqnarray*}
Thus letting $C = 2p$,  another simple calculation for
 each $2 \leq j \leq d$ yields
\begin{equation}\label{part2}
B k_j + B k_{j+d} + C k_{d+1} = 0,
\end{equation}
and so summing \eqref{part2} for $j=2$ to $d$ and adding \eqref{part1}
yields
\begin{equation*}
A k_{d+1} + B k_1 
+ C (d-1) k_{d+1}  
+ \sum_{j = 2}^d ( B k_j + B k_{j+d})
= 0.
\end{equation*}
Thus as noted in the second paragraph of the proof this implies,
\begin{equation*}
H  := \gen{B k_1, B k_2, \dots, B k_d} 
\subset \genp{g_1 + w, g_2 + w, \dots, g_{2d} + w}.
\end{equation*}

Finally, to show that $\{B k_1, B k_2, \dots, B k_d\}$
is linearly independent, it suffices to
show that  $\{ k_1, k_2, \dots, k_d\}$ is linearly
independent over $\Q$. If 
$\sum a_j k_j = 0$, then using the definition of
the $k_j$, 
\begin{equation*}
0 = (a_1 (n_0 + 1 + p/q) + p/q(a_2 + \dots + a_d)) \barg_1
+ a_2 g_2 + \dots a_d g_d,
\end{equation*}
and so all $a_j = 0$ using the linear independence of
$\{\barg_1, g_2, \dots, g_d\}$. \QED 
\end{proof}

\subsection{Nilpotent linear Transformations}
In the sufficient conditions for transitivity given
in Proposition~\ref{MSP} below the twisting matrix is required
to satisfy $\spec{A} = \{1\}$. It is often technically
convenient to work with the nilpotent
matrix $A-\id$ in such cases. 
A linear transformation $T:\Z^d \raw \Z^d$ is said
to be \de{nilpotent of order $J$} if $T^{J} = 0$ and
$T^{J-1}\not= 0$. The same definition applies
to square matrices. A nilpotent transformation
of order $J$  generates
a chain of kernels,
\begin{equation*}
0 = \ker(T^0) \subset \ker(T) \subset \ker(T^2) \subset
\dots \subset \ker(T^J) = \Z^d.
\end{equation*}
It is also easy to check that 
\begin{equation}\label{eq2}
T(\ker(T^j)) \subset
\ker(T^{j-1}) \subset \ker(T^j).
\end{equation}
\begin{lemma}\label{NL}
Let $T:\Z^d \raw \Z^d$ be a nilpotent homomorphism
of order $J$. Then there is a direct sum
decomposition  $\Z^d =  V_1 \osum \dots\osum V_J$ with each
 $V_j \not= 0$ so that  for all $j$,
\begin{compactenum}
\item $\ker(T^j) = V_1 \osum \dots V_j$,
\item $T(V_j)\subset V_1 \osum \dots V_{j-1}$,
\item If $p_j:\Z^d\raw V_j$ is the projection,   
then $ (p_{j-1} \circ T)_{\vert V_{j}}$ is injective for $j>1$.
\end{compactenum}
\end{lemma}

\begin{proof}
The proof proceeds by induction on the order of nilpotency
$J$ with the case $J=1$ being trivial. 
Assume then that the result is true for all nilpotent transformations
of order less than $J$ defined on any finite rank free
Abelian group.

For simplicity of notation let $K := \ker(T^{J-1}) $ and
$\hT = T_{\vert K}$. By \eqref{eq2}, $\hT(K)\subset K$, and
since $\hT$ has order of nilpotency $J-1$, by the inductive
hypothesis, there is a direct sum decomposition  
 $K =  V_1 \osum \dots\osum V_{J-1}$ with each
 $V_j \not= 0$ so that conditions (a), (b), and (c) hold
  for all $1\leq j \leq J-1$.

Since $K$ is a kernel, one easily sees that 
it is a pure subgroup of
$\Z^d$, and thus for some  subgroup $V_J\subset \Z^d$,
$\Z^d = K \osum V_J$.  We now
check that  $\Z^d =V_1 \osum \dots\osum V_J$ satisfies the
required conditions (a), (b), and (c) 
 for all $1\leq j \leq J$

Since $\ker(T^J) = \Z^d$, condition (a) is obviously
satisfied. 
Now certainly, $V_J \subset \ker(T^J) = \Z^d$, and
so by \eqref{eq2}, 
\begin{equation*}
T(V_J) \subset T(\ker(T^J)) 
\subset \ker(T^{J-1}) = V_1 \osum \dots\osum V_{J-1},
\end{equation*}
and so condition (b) is satisfied.

Now we confirm condition (c). We have just shown that
$T(V_J) \subset  V_1 \osum \dots\osum V_{J-1}$. Thus if
$v\in V_J$ and $p_{J-1}\circ T(v) = 0$, then in fact
$T(v) \subset  V_1 \osum \dots\osum V_{J-2} = \ker(T^{J-2})$,
and so $T^{J-2}(T(v)) = 0$, and so $v\in \ker(T^{J-1}) = K$ and
therefore $v \in K \cap V_J = 0$. Thus $\ker(p_{J-1}\circ T_{\vert V_J}) 
=\{0$\}, as required. 
\QED\
\end{proof}

\subsection{Integer matrices with spectrum equal to $\{1\}$}
For a linear transformation  $T$
or a square, integer matrix $A$, its \de{spectrum} 
 is denoted $\spec(T)$ or $\spec(A)$.
We will need a notation for the  block description of a matrix.
Given a collection of positive integers  
$n_1, n_2,  \dots, n_J$ with $\sum n_\alpha = d$, the
block description of type $(n_1, n_2, \dots, n_J)$
of the  $d\times d$ matrix $A$ consists of the
matrices $B_{\alpha,\beta}$, with $1\leq\alpha,\beta\leq J$
with the dimension of $B_{\alpha,\beta}$ equal to 
$n_\alpha\times n_\beta$ and $A = (B_{\alpha,\beta})$, or
more explicitly, the $(i,j)^{th}$ entry of $B_{\alpha, \beta}$ is
the $(n_1 + \dots + n_{\alpha-1} + i, n_1 + \dots + n_{\beta-1} + j)^{th}$
entry of $A$.

The next lemma uses Lemma~\ref{NL} to give a special form for 
a matrix $A$ representing a linear isomorphism
 $S$ with 
$\spec(S) = \{1\}$. The simplest case of a matrix $A$ in
the form is lower tridiagonal with
all $1$'s on the diagonal and all the entries in the first
subdiagonal are nonzero.

\begin{lemma}\label{CFL}
If $S:\Z^d\raw \Z^d$ is an automorphism with $\spec(S) = \{1\}$,
then there a collection of numbers
 $n_1 \geq n_2 \geq \dots \geq n_J > 0$ with $\sum n_j =d$
 and a basis of $\Z^d$ such that
with respect to this basis the automorphism
$S$ is represented by
a matrix $A$ which when written in block form
 of type $(n_1, n_2, \dots, n_J)$ 
has blocks $B_{\alpha,\beta}$ satisfying 
\begin{compactenum}
\item $B_{\alpha,\alpha} = I$ for $\alpha = 1, \dots, J$,
\item $B_{\alpha,\beta} = 0$ for $\alpha < \beta$,
\item $B_{\alpha,\alpha-1}$ has rank $n_\alpha$
 for  $\alpha = 2, \dots, J$.
\end{compactenum}
\end{lemma}

\begin{proof} Let $C$ be the matrix representing $S$ in
the standard basis for $\Z^d$. Since by hypothesis
$\spec(C) =  \{1\}$, we have 
$\spec(C^T) =  \{1\}$ and so $C^T - I$ represents
a nilpotent homomorphism which we denote $T$.
 Applying Lemma~\ref{NL} to $T$, we find
a basis for $T$ or equivalently a unimodular matrix $E$, so
that $\bar{A} := E (C^T - I) E\I$ is block factored of type
 $(n_1, n_2, \dots, n_J)$, where $n_j = \rank(V_j)$ with the $V_j$
as in Lemma~\ref{NL} and 
\begin{compactenum}
\item $\bar{B}_{\alpha,\alpha} = 0$ for $\alpha = 1, \dots, J$,
\item $\bar{B}_{\alpha,\beta} = 0$ for $\alpha > \beta$,
\item $\bar{B}_{\alpha-1,\alpha}$ has rank $n_\alpha$ 
for  $\alpha = 2, \dots, J$.
\end{compactenum}
Finally,  let $A:=\bar{A}^T + I :=   (E^T)\I (C - I) (E^T)  + I = 
(E^T)\I C (E^T)$, and so $A$ represents $S$ 
and it is easy to check that it has
the required block form using the block form of $\bar{A}$.
\QED
\end{proof}

\subsection{Behavior of the Fried quotient under iteration}\label{FFsect}
For any quotient twisted skew product $\tau_\Gamma$
 it follows easily that  $(\tau_\Gamma)^k =
 (\tau^k)_\Gamma$. However, for $k>0$  letting
 $F^{(k)} := P( im(\Psi^k - \id))$, 
 in general one has  $F^{(1)} \not =F^{(k)}$. A simple example
is $\Psi = -\id$.  
  Thus the iterate of the Fried quotient and the
Fried quotient of the iterate, 
$\tau_F^k$ and $(\tau^k)_{F^{(k)}}$,
often act on different spaces and
thus are not equal. However, in the special case of
$\spec(\Psi) = \seto$ we have as a corollary of Lemma~\ref{CFL}:

\begin{corollary}\label{FFRmk}
If $\tau$ is a twisted skew product with twisting
matrix $A$ satisfying $\spec(A) = \seto$,
then the iterate of the Fried quotient is the
Fried quotient of the iterate, or $(\tau_{F^{(1)}})^k=
 (\tau^k)_{F^{(k)}}$, where $F^{(k)} := P( im(\Psi^k - \id))$.
\end{corollary}

\begin{proof}
Since $\spec(A) = \seto$, we may conjugate $A$ so
that it is in the form given by the block factorization
of Lemma~\ref{CFL}. In the proof of that lemma,
this block factorization corresponds to an internal direct sum
decomposition $\Z^d = W_1 \osum \dots \osum W_J$ as in
Lemma~\ref{NL}. As a consequence of condition (c) in 
Lemma~\ref{CFL}, the purification of $\im(A - \id)$
is exactly $F = W_2 \osum \dots \osum W_J$, and so
$\Z^d/F$ is naturally identified with $W_1$.

Denoting the block factorization 
of $A$  as $B_{\alpha,\beta}$, then for $k>0$, a 
simple computation shows  that we can get  such a factorization
for $A^k$ by taking the $k^{th}$ power of the block
factorization of $A$. Specifically, if we denote 
 the factorization so obtained for $A^k$ as 
$B_{\alpha,\beta}^{(k)}$, then subdiagonal blocks satisfy
$B_{\alpha + 1,\alpha}^{(k)} = k\;B_{\alpha + 1,\alpha}$
 for  $\alpha = 1, \dots, J-1$. Thus since $B_{\alpha + 1,\alpha}$
has rank $n_\alpha$, so does $B_{\alpha + 1,\alpha}^{(k)}$. This implies
that the purification of $\im(A^k - \id)$ is 
 $F^{(k)} = W_2 \osum \dots \osum W_J$ for all $k>0$.
Thus for all $k>0$,  $\Z^d/F^{(k)} = \Z^d/F$
which implies that $(\tau^k)_{F^{(k)}} = (\tau_{F^{(1)}})^k$.
\QED\
\end{proof}

\section{Transitivity of Twisted Skew Products}

\subsection{The main induction lemma}
When the twisting matrix is 
a generalized shear, transitivity is proved by treating
the twisted skew product as a sequence of
untwisted extensions. The main technical step in
doing this is given in the next lemma. 

The simplest nontrivial case of the lemma is
when  $\tau$ has group component $\Z^2$, height
function $h = (h_1, h_2)$, and
twisting matrix 
\begin{equation}\label{skewmat}
A = \begin{pmatrix}
1 & 0\\
1 & 1
\end{pmatrix}.
\end{equation}
Thus in coordinates, 
\begin{equation*}
\tau(\us, m, n) = (\sigma(\us), m + h_1(\us), n + m + h_2(\us)).
\end{equation*}
We then define $\eta$ on $\Sigma\times\Z$ as 
$\eta(\us, m) = (\sigma(\us), m + h_1(\us))$. As in
\S\ref{conn} we treat $\eta$ as a countable state Markov
shift on $\Sigma' = \Sigma\times \Z$. After letting
$\ut = (\us, m)$ and $g(\ut) = g(\us, m) = m + h_2(\us)$, we may
write 
\begin{equation*}
\tau(\ut, n) = (\eta(\ut), n + g(\ut)),
\end{equation*}
and so $\tau$ treated as a map on $\Sigma'\times\Z$ 
is an \textit{untwisted} extension of $(\Sigma',\eta)$. In this
case the lemma says that if $\eta$ is transitive and
$\tau$ has the ftp, then $\tau$ is transitive. Note
that no separate hypothesis in required in the shearing
direction. Roughly speaking, the shear creates a global circulation
which is  conducive of recurrence.

The  lemma allows  group components
$\Z^k$ and $\Z^\ell$, and it requires that the twisting
matrix $A$ be $(k, \ell)$-block factored in lower triangular form
with identity matrices on the diagonal and the
subdiagonal block having full rank.

\begin{lemma}\label{MIL}
Let $\eta:\Sigma\times\Z^k\raw \Sigma\times\Z^k $ 
be a transitive, untwisted skew
product with  height function $h$ and
base shift $(\Sigma,\sigma)$ countable Markov.
 Let $\Sigma' = \Sigma\times\Z^k$ 
and assume that $\tau:\Sigma'\times\Z^\ell\raw \Sigma'\times\Z^\ell$
is an untwisted skew product 
given by
\begin{equation}\label{taudef}
\tau(\ut, n) = (\eta(\ut), n + g(\ut))
\end{equation}
and for $\ut = (\us,m) \in \Sigma'$, the
height function $g$ is required to have the form
\begin{equation}\label{gdef}
g(\ut) = g(\us,m) = S(m) + f(\us),
\end{equation}
with $f:\Sigma\raw\Z^\ell$  constant on
length two cylinder sets and $S:\Z^k\raw \Z^\ell$
a rank $\ell$-homomorphism. If $\tau$ has the ftp,
then it is transitive.
\end{lemma}

\begin{proof}
Since $\eta$ is transitive, it certainly has a periodic point,
say $\utn = (\us', 0)$ of period $n_0$. Now since $\eta$ is
untwisted, $\utm := (\us', m)$ is also a periodic orbit
of period $n_0$ for $\eta$.  We treat $\utm\in \Sigma'$
and a straightforward computation using \eqref{taudef}
and \eqref{gdef}
 yields that
\begin{eqnarray*}
g(\utm, n_0) &=& g(\utm) + g(\eta(\utm)) + \dots + g(\eta^{n_0-1}(\utm))\\
&=& g(\us', m) + g(\sigma(\us'), m + h(\us')) + 
\dots g(\sigma^{n_0-1}(\us'), m + h(\us', n_0 -1)) \\
&=&  n_0 S(m) + S\left(h(\us') + \dots h(\us', n_0 -1)\right)
+ f(\us', n_0).
\end{eqnarray*} 
Thus letting $w =  S\left(h(\us') + \dots h(\us_0, n_0 -1)\right)
+ f(\us', n_0)$, we have that $\im(n_0 S) + w \subset D(\tau)$.
Since $S$ has rank $\ell$ by hypothesis, $\im(n_0 S)$ is a rank
$\ell$ subgroup of $\Z^\ell$. Thus  we may apply
 Lemma~\ref{GTL} to  $\im(n_0 S)$ yielding elements
 $m_1, \dots, m_{2d}\in\Z^k$ and
 a rank $\ell$-subgroup $H$, with
\begin{equation*}
H\subset  \genp{n_0 S(m_1) + w, \dots, n_0 S(m_{2d})+w}
\subset\genp{D(\tau)}.
\end{equation*}
Thus by Lemma~\ref{ftp2}, since $\tau$ has the ftp
by assumption, it is transitive.
\QED\
\end{proof}

\subsection{Sufficient conditions for transitivity}
The next proposition uses Lemma~\ref{MIL} as  the induction
step to get transitivity in the case when the
twisting automorphism is a generalized shear in
the form given by Lemma~\ref{CFL}.

\begin{proposition}\label{MSP}
 Assume that 
$\tau:\Sigma\times\Z^d\raw \Sigma\times\Z^d$ is a twisted
skew
product with  height function $h$ and
base shift $(\Sigma, \sigma)$  countable Markov.
 If 
 the twisting automorphism $\Psi$ is 
a generalized shear ($\spec(\Psi) =\{1\}$), 
the  Fried quotient $\tau_{F}$ is transitive,
 and $\tau$ has the ftp, 
then $\tau$ is transitive.
\end{proposition}

\begin{proof} 
Since $\spec(\Psi) =\{1\}$, using  Lemma~\ref{CFL}
we can start by choosing a basis for $\Z^d$  
in which $\Psi$ is represented by a matrix $A$ with block
factorization $B_{\alpha,\beta}$ for $\alpha,\beta = 1, \dots,
J$ as in that lemma. Further, we adopt the notation of  Lemma~\ref{NL} 
by assuming that the block factorization of $A$ corresponds to
the internal direct sum decomposition
$\Z^d = W_1 \osum \dots W_J$ with $\rank(W_j) = n_j$.
If $p_j:\Z^d \raw W_j$ is the projection, let $h_j = p_j\circ h$.

We now define a collection of spaces and
untwisted skew products inductively. Let $\Sigma^{(1)} = \Sigma$,
and $\tau_1:\Sigma^{(1)}\times W_1\raw \Sigma^{(1)}\times W_1$
be defined as $\tau_1(\us, m_1) = (\sigma(\us), m_1 + h_1(\us))$.
For $j = 2, \dots, J$, let
 $\Sigma^{(j)} = \Sigma^{(j-1)}\times W_{j-1}$ and
 $\tau_j: \Sigma^{(j)}\times W_j \raw  \Sigma^{(j)}\times W_j$
is defined using $\ut^{(j-1)} \in \Sigma^{(j)}$ written
as  $\ut^{(j-1)} = (\us, m_1, \dots, m_{j-1})$ by
\begin{equation*}
\tau_j(\ut^{(j-1)}, m_j)
 = \left(\tau_{j-1}(\ut^{(j-1)}), m_j + H_j(\ut^{(j-1)})\right)
\end{equation*}
where
\begin{equation*}
H_j(\ut^{(j-1)}) = 
B_{j, 1} m_1 + B_{j, 2} m_2 + \dots + B_{j, j-1} m_{j-1} + h_j(\us).
\end{equation*}

Note that $\tau_J = \tau$, the given twisted
skew product. Also, if for
$j = 1, \dots, J-1$,  we let $\Gamma_j = W_{j+1} \osum\dots W_J$,
then $\Psi(\Gamma_j) \subset \Gamma_j$ by the form of $A$, and
$\tau_j$ is exactly the quotient map 
$\tau_{\Gamma_j}$ as defined in \S\ref{quotientsect}.
Thus since by hypothesis $\tau$ has the ftp,
each $\tau_j$ also has the ftp and the base shift is
transitive. As noted in the proof of Corollary~\ref{FFRmk},
$\tau_1$ is the Fried quotient of $\tau$ and
so is transitive by hypothesis. Finally, 
by the block factorization of $A$ obtained using Lemma~\ref{CFL}
each $ B_{j, j-1}$ is rank $n_j$, and so each $\tau_{j}$ extends $\tau_{j-1}$ 
as required for Lemma~\ref{MIL}. Thus by induction each $\tau_j$
is transitive, and so $\tau_J = \tau$ is also.
\QED\
\end{proof}

\subsection{The main symbolic theorem}
Before we state our main theorem about the transitivity
of twisted skew products we  recall a few more
definitions and facts about an automorphism $\Psi$ of a finitely
generated free Abelian groups or equivalently, the
matrix that represents it in some basis. For simplicity
we just give the terminology and notation for the matrices.

 The \de{spectral radius} of $A$
is denoted  $\rho(A)$. Since matrix $A$ is invertible 
if and only if $\det(A) = \pm 1$, when $A$ is
invertible, if $\rho(A)\not = 1$, there must be eigenvalues
of modulus both larger and less than one. If 
 $\rho(A) = 1$, the eigenvalues of $A$ must all lie on
the unit circle. Now for an eigenvalue $\lambda$ of $A$, certainly
its minimal polynomial must be a factor of the characteristic
polynomial of $A$. Thus if $\rho(A) = 1$,
 all the algebraic conjugates of any eigenvalue $\lambda$ must lie
on the unit circle. Thus as a consequence of a theorem of Kronecker
 (see, for example, problem 11 on pg. $40$ in \cite{marcus}),
$\lambda$ must be a root of unity (the author learned this from
Peter Sin whom he acknowledges with gratitude).  Thus we see that
if $A$ is invertible and $\rho(A) = 1$, then for some integer $N>0$, 
$\spec(A^N) = \{1\}$. This allows the following definition.
\begin{definition}\label{Ndef}
If $A\in\SL(d, \Z)$ and has spectral radius equal to
one, $\rho(A) = 1$, let $N(A)$ be the least positive 
integer with $\spec(A^{N(A)}) = \seto$.
\end{definition}

While many of the preceding results involved
skew products with the base shift a countable Markov
chain,  the main theorem concerns the case where the base shift
a transitive subshift of finite type. This restriction is
required in order to use the various properties
of the rotation set from \S \ref{rotsub}. It is worth remarking,
 however, that the
proof itself uses an induction with untwisted 
extensions of countable state Markov shifts. 
It is also worth remarking that since the alternative (b) in the
theorem only requires that the twisting matrix $A$ has spectral radius one,
we must consider $\tau^{N(A)}$ in order to ensure that the
Fried quotient is nontrivial and the rotation set is thus defined.
Alternative (b) also contains a condition that is the
analog of totally transitive for finite quotients; we say that
$\tau$ has the \de{total ftp} if $\tau^k$ has the ftp for all $k>0$

\begin{theorem}\label{MST}
 Assume that 
$\tau:\Sigma\times\Z^d\raw \Sigma\times\Z^d$ is a twisted
skew product with twisting matrix $A$, height function $h$, and
base shift $(\Sigma, \sigma)$ which is a transitive subshift
of finite type. The following are equivalent:
\begin{compactenum}
\item $\tau$ is totally transitive,
\item $\rho(A) =1$, $\tau$ has the total ftp, 
and $0\in Int(\rot_{F}(\tau^{N(A)}))$,
\item $\tau$ has the total ftp and its periodic points are
dense in $\Sigma\times\Z^d$,
\item $\tau$ is topologically mixing
\end{compactenum}
\end{theorem}

\begin{proof}
We first show that (b) implies (a).
 For simplicity of notation, let $\eta_k = \tau^{k N(A)}$.
Using \eqref{skewiteq} the twisting matrix of $\eta_k$ is $A^{k N(A)}$. 
By definition $\spec(A^{N(A)}) = \{1\}$,
and so for all $k>0$, $\spec{A^{k N(A)}} = \{1\}$ also.
 Since by hypothesis  $\tau$ has the total ftp, each
$\eta_k$ also has the ftp.
We now show that the Fried quotient of  each $\eta_k$ is transitive, 
and then Lemma~\ref{MSP} will give that each $\eta_k$ is transitive. . 

 Corollary~\ref{FFRmk} and \eqref{rotpower} yield that 
\begin{equation*}
\rot\left((\tau^{k N(A)})_{F^{(k N(A))}}\right) = 
\rot\left(((\tau^{ N(A)})_{F^{(N(A))}})^k\right)
=k\rot\left((\tau^{ N(A)})_{F^{(N(A))}}\right).
\end{equation*}
Thus since $0\in Int(\rot_{F}(\tau^{N(A)}))$ by hypothesis, 
 $0\in Int(\rot_{F}(\eta_k))$
 for all $k>0$. 
As already noted, each $\eta_k$ has the ftp
 and so as remarked in \S\ref{ftpdef}, 
 $(\eta_k)_{F}$ has the ftp for all  $k>0$. Thus
since each $(\eta_k)_{F}$ is untwisted,
by Theorem~\ref{rotthm}, each $(\eta_k)_{F}$ is transitive,
finishing the proof that each $\eta_k$ is transitive.
Now since  $\eta_k = (\tau^{k})^{ N(A)}$,
a power of every $\tau^k$ is transitive, and so each  
$\tau^k$ is also transitive,  proving (a)

Now conversely, assume that $\tau$ is totally transitive.
If $\rho(A) \not = 1$, there must be an eigenvalue $\lambda$
with $|\lambda|>1$. Assume first that $\lambda$ is real and
positive and treating $A$ as acting on $\R^d$,
let $v_1\in\R^d$ be an eigenvector corresponding to $\lambda$.
Extend $v_1$ to a basis for $\R^d$ and for $w\in\R^d$, let
$\Phi(w)$ be its first component with respect to this
basis, and so $\Phi(A^k w) = \lambda^k \Phi(w)$. 

Since by hypothesis $(\Sigma,\sigma)$ is a subshift
of finite type, $h$ is bounded and let $|\Phi\circ h| < C$.
Now,
\begin{equation}\label{twistiterates}
\tau^k(\us,\vn) = \left(\sigma^k(\us), A^k \vn + A^{k-1} h(\us) + \dots 
+ A h(\sigma^{k-2}(\us)) + h(\sigma^{k-1}(\us))\right),
\end{equation}
and so 
\begin{equation}\label{geo}
\begin{split}
|\Phi\circ\pi_2(\tau^k(\us,\vn))| &\geq \lambda^k |\Phi(\vn)| 
- (\lambda^{k-1} + \dots \lambda + 1) C\\
&= \lambda^k\left(|\Phi(\vn)|-\frac{C}{\lambda-1}\right) +\frac{C}{\lambda-1}.
\end{split}
\end{equation}
Thus if $\vn_0$ is such that  $|\Phi(\vn_0)|>\frac{C}{\lambda-1} $,
then $|\Phi\circ\pi_2(\tau^k(\us,\vn_0))|\raw\infty$ as 
$k\raw\infty$ and so
no point in the open set $\Sigma\times \{\vn_0\}$ can have
a dense, forward $\tau$ orbit, and so $\tau$ is not transitive. The
case where the eigenvalue $\lambda$ is complex is similar,
but now one uses  a $\Phi$ that projects onto the two-dimensional
subspace associated with $\lambda$ in the real Jordan form.
Thus $\tau$ being transitive implies that $\rho(A)=1$,
Finally, if $\tau$ is totally transitive, then 
certainly any quotient of any power is also transitive and so
using Theorem~\ref{rotthm},  $\tau$ has the total ftp 
and $0\in Int(\rot_{F}(\tau^{N(A)}))$,
finishing
the proof that (a) implies (b).

Now note that the argument just given shows that
shows that if $\rho(A) \not= 1$, then the recurrent points of
$\tau$ cannot be dense in $\Sigma\times\Z^d$. 
Thus assuming (c), we have 
$\rho(A) = 1$. In addition, if $\tau$ has dense periodic points then so
does any power or quotient, so in particular, $(\tau^{N(A)})_F$
has dense periodic points and is untwisted and has the ftp by construction.
Thus by Theorem~\ref{rotthm},  $0\in Int(\rot_{F}(\tau^{N(A)}))$ proving
that (c) implies (b).

The fact that transitivity implies dense periodic points was
noted in the proof of Theorem~\ref{rotthm}, and so (a) implies (c). The equivalence
of (a) and (d) for countable state Markov shifts 
is standard and was noted in \S\ref{Markovdef}, finishing the proof.
\QED\
\end{proof}

\begin{remark}\label{MSTrk}
For future reference we note that the argument above based
on \eqref{geo} yields that 
if $\tau$ is transitive, then $\rho(A) = 1$. 
In addition, an analogous argument 
 gives that $|\Phi(\vn - \vn')| > \frac{2C}{\lambda-1}$
implies that for any 
$\us, \us'\in\Sigma$, 
\begin{equation*}
 \left|\Phi\circ\pi_2(\tau^k(\us,\vn)) -
\Phi\circ\pi_2(\tau^k(\us',\vn'))\right|\raw\infty,
\end{equation*} 
as $k\raw\infty$, where $\pi_2:\Sigma\times\Z^d\raw\Z^d$ is the projection.
\end{remark}

\section{Lifted Rel Pseudo-Anosov Maps}
With the help of Theorem~\ref{conn} we now apply 
Theorem~\ref{MST} on twisted skew products to the
study of transitive lifts of rel \pA\ maps.

\subsection{Necessary and sufficient conditions for total transitivity}\label{leafdef}
We recall some definitions about the invariant singular
measured foliations, $\cF^u$ and $\cF^s$, of a rel \pA\ map.
 Because of the presence of  a finite number of 
singularities, there are several ways to define
a ``leaf'' of the foliation. Fix one foliation, say  $\cF^u$,  
and let $P$ denote the set of its singularities union the boundary
components of $M$. On the punctured surface $M-P$,  
$\cF^u$ is a non-singular foliation. 
Points  $x\in M-P$ are called \de{regular points} and  
the leaf of $\cF^u$ 
containing these points is the leaf of the foliation on $ M-P$.
Thus the leaves containing regular points are either immersed lines 
in $M$ or else immersed half-lines which ``begin'' at a singularity. 
In the latter case the singular point is said to be \de{associated with the
leaf}. The leaves which contain non-regular points are 
called \de{trivial leaves}. These are of two types:
a singular point is consider a trivial leaf  as are
any leaves which are segments contained in the boundary of $M$.
We also recall the dynamical meaning of the invariant foliations.
Given a topological metric $d$ on $M$, 
\begin{equation}\label{dynmeaning}
d(\phi^n(x_1), \phi^n(x_2))\raw 0\ 
\end{equation}
as $n\raw\infty$ (respectively, $n\raw-\infty$)  if and only if 
$x_1$ and $x_2$ are on the same leaf of $\cF^s$ ($\cF^u$),
or their leaves are associated with the same singularity. 

In the universal Abelian cover  the singular foliations
lift to a pair $\tcF^u$ and $\tcF^s$.
For a point $\tx$ in $\tM$ its leaf of $\tcF^u$ is defined
as in the base, or equivalently, project $\tx$ to $x\in M$,
find the leaf of $x$ and then the leaf of $\tx$ is the lift of this
leaf in $\tM$ which contains $\tx$. 
The analog of \eqref{dynmeaning} also holds
using any equivariant metric $\td$.

\begin{theorem}\label{pA1}
Let $\phi:M^2\raw M^2$ be rel \pA, 
$\tM$ is the universal Abelian cover, 
and $\tphi$ is a lift of $\phi$ to $\tM$.
The following are equivalent:
\begin{compactenum}
\item $\tphi$ is totally transitive,
\item $\rho(\phi_*) = 1$ and $0\in Int(\rot_{F}(\tphi^{N(\phi_*)}))$, 
\item The set of periodic  points of $\tphi$ is dense
in $\tM$,
\item $\tphi$ is topologically mixing,
\item There is a periodic regular point
$\tx$ of $\tphi$ so that the leaf of the 
lifted foliation $\tcF^u$ containing $\tx$ is dense in
$\tM$. 
\end{compactenum}
\end{theorem}

\begin{proof}
Let $\tau$ be a twisted skew product corresponding to $\tphi$,
and let $\talpha$ be the semiconjugacy given in \S\ref{semiconstruct}.
The equivalence of (a), (b), (c), and (d) 
 follows from Proposition~\ref{summary}
and Theorem~\ref{MST}. 

Now assume (a)-(d) and let $\tx\in\tM$ be a periodic point
for $\tphi$ with period $n_0$. 
Since the collection of singular points is a discrete set 
in $\tM$ and periodic points of $\tphi$ are dense, we
may assume that $\tx$ is a regular point.
By hypothesis $\tphi^{n_0}$ and thus $\tau^{n_0}$ have  dense forward orbits
using Theorem~\ref{bairetrans}.  Treating $\tau$ as a countable state Markov
shift, let $\hht = \dots \hht_{-1} \hht_0\hht_1\dots$
be a point whose forward orbit is dense under $\tau^{n_0}$.

Let $\ut = (\us,\vn)\in \Sigma\times\Z^d$ be a period-$n_0$
point of $\tau$ with $\talpha(\us,\vn) = \tx$. 
 If the periodic
block of $\ut$ is $b = b_0 \dots b_{n_0-1}$, since $\tau$ is transitive
we have an allowable $b_0\raw \hht_0$.  Call this block $c$. 
Now form the sequence $w = \dots b\, b\, b\, c\,  \hht_0\, \hht_1 \dots $.
By construction, 
$\talpha(o_+(w,\tau^{n_0})) = o_+(\talpha(w), \tphi^{n_0})$ is 
dense in $\tM$. In addition, $\tau^{-{n_0}k}(w)\raw \ut$ as
 $k\raw\infty$ and so $\tphi^{-{n_0}k}(\talpha(w))\raw \tx$.
Thus if $\tL$ is the leaf of $\tcF^u$ which contains $\tx$,
since $\tx$ is a period $n_0$ point we have $\talpha(w)\in \tL$ and so 
$o_+(\talpha(w), \tphi^{n_0})\subset \tL$ as well, and so
$\tL$ is dense in $\tM$, finishing the proof  that
(a)-(d) implies (e).

Now assume (e) and let $\tx$ be a regular 
periodic point with  period $n_0$ and its unstable leaf 
$\tL$ dense in $\tM$. We show that (b) follows
by first showing that  $\rho(\phi_*) = 1$. 
Assume to the contrary that $\rho(\phi_*)\not=1$. 
Since $\det(\phi_*) = 1$, this implies that
that $\phi\I_*$ has an eigenvalue $\lambda$ with $|\lambda|> 1$. 
Then, just as in the proof of Theorem~\ref{MST} and
using \eqref{linearcompare}
there is a linear functional
$\Phi:\R^n\raw\R$ and a constant $C>0$ so  that
$\Phi(\tbeta(\ty)) > C/(\lambda-1)$ implies that 
$\Phi(\tbeta(\tphi^k(\ty)))\raw\infty$ as $k\raw-\infty$,
where $\tbeta:\tM\raw\R^d$ is the map constructed in
\S\ref{coverrot}. 

Now since $\cL$ is an unstable leaf containing the
periodic  regular point $\tx$, if $I$ is a small segment in $\cL$
with $\tx$ in its interior, then 
$\cap_{j>0}\tphi^{-n_0 j}(I)= \tx$ and 
$\cup_{j\geq 0}(\tphi^{n_0 j}(I)) = \cL$. Since $\cL$ is
dense in $\tM$ there are certainly 
$\ty\in \cup_{j\geq 0}(\tphi^{n_0 j}(I))$ with
$\Phi(\tbeta(\ty)) > C/(\lambda-1)$. Thus,  
$\Phi(\tbeta(\tphi^n(\ty)))\raw\infty$ as $n\raw -\infty$, but
$\tphi^{-n_0 j_0} (\ty)\in I$ for some $j_0 > 0$ implies
that $\tphi^{n_0 j} (\ty)\raw \tx$ as $j\raw -\infty$, a contradiction, 
finishing the proof that  $\rho(\phi_*) = 1$.
 
As in Definition~\ref{Ndef}, let 
$N(\phi_*)$ be such that $\spec(\phi_*^{N(\phi_*)}) = \seto$. If
$\tM'$ is the Fried cover of $\tphi^{N(\phi_*)}$,
$\tM' = \tM/F^{(N(\phi_*))}$ in
the notation of \S\ref{FFsect}.
Let $\chi:\tM\raw\tM'$ be the
projection and $\tphi'$ be the projection of $\tphi$ to
$\tM'$. 
Let $\tx' = \chi(\tx)$ where $\tx$ is a period-$n_0$
regular point with  its unstable leaf $\cL$ dense in $\tM$. 
If $\cL'$ is the unstable leaf in $\tM'$ which contains
$\tx'$, then certainly $\cL' = \chi(\cL)$, and
so $\cL'$ is dense in $\tM'$.
 Now as above let 
$I'\subset \cL'$ be a small segment with $\tx'$ in its interior and so
$\cup_{j>0}((\tphi')^{n_0 j}(I'))$ is dense in $\tM'$. 

Now let
$\tau'$ be the Fried quotient of $\tau^{N(\phi_*)}$ and $\talpha'$ the
induced semiconjugacy from $\tau'$ to $\tphi'$.
Since $\cup_{j>0}((\tphi')^{n_0 j}(I'))$ is dense in $\tM'$,
using that fact that $\talpha$ is a quasi-isometry as 
observed in \S\ref{quasisect}, certainly if $h'$ is height function
of $\tau'$, for every onto linear functional $L$,
the corresponding supremum in \eqref{nobound} is
infinite. Thus by Lemma~\ref{boundlem} and using Lemma~\ref{summary},
$0\in Int(\rot(\tphi')) = Int(\rot_{F}(\tphi^{N(\phi_*)}))$, completing the
proof.
\QED\
\end{proof}

\begin{remark}
When $G\subset H_1(M)$ with $\phi_*(G) = G$ and $H_1(M)/G$ torsion-free,
the analog of Theorem~\ref{pA1} for the cover $\tM_G$ has an almost
identical statement and proof.  The case where $H_1(M)/G$ has torsion
requires a more elaborate statement of condition (b), and we
leave it to the interested reader.
\end{remark}

\begin{remark}
Rel \pA\ maps on the torus which act on homology by 
the skew matrix $A$ in \eqref{skewmat} are 
the simplest examples of rel \pA\ maps which satisfy the hypothesis
of the theorem with $\phi_*\not=\id$.
Rel \pA\ maps in  this class were studied in \cite{doeff}, 
\cite{doeffmis}, and \cite{tali}. 
The general notion of $\rot_F$ when restricted
to this torus shear case was called the shear rotation interval
by Doeff, and many of the basic properties
as in \S\ref{rotsub} above were proved.  
A fascinating explicit example of map $\tpsi$ on the plane
which is the lift of a rel \pA\ map  in this class 
 was given by \cite{CG} and analyzed by 
\cite{mackay}. In the example $\rot_F(\tpsi) = [0,1]$,
and so by Theorem~\ref{pA1}, the example is not
transitive on the plane. However, for example, the theorem
does imply that 
$\delta_{(-1,0)}\circ \tpsi^2 $ is transitive.
\end{remark}

\begin{remark} When $\phi$ is  \pA\ rel a nonempty finite
set, its foliations have one-prongs and are thus non-orientable.
A true \pA\ map has oriented foliations if and only if
$\rho(\phi_*)$ is equal to its dilation $ \lambda > 1$.
(see, for example,
Lemma 4.3 in \cite{BB0}). Thus in the situation
of Theorem~\ref{pA1} where $\rho(\phi_*) = 1$, the invariant
foliations $\cF^u$ and $\cF^s$ are always non-orientable.
\end{remark}  

\begin{remark} The characterization of transitive twisted
skew products given in Theorem~\ref{MST} allows the twisting
matrix to be any $A\in\SL(d, \Z)$. For a surface \homeo\ 
$f$ whose lift is modeled by the skew product, $f_* = A$ 
has additional structure. When the surface is closed,
$f_*$ is symplectic and the addition of boundary components
only gives rise to permutations on $H_1(M)$. Thus the surface
dynamics applications do not require the full force of 
Theorem~\ref{MST}.
\end{remark}

\subsection{The case when $\phi_* = \id$.}\label{idcase}
Most of the literature on the dynamics of lifted
maps concerns the case of maps
isotopic to the identity. In this as well as
the more general case where $\phi$ acts trivially
on first homology, condition (b) in Theorem~\ref{pA1} is
replaced by the simpler condition $ 0\in Int(\rot(\tphi))$.

In this case there is a fair amount known about dynamical
representatives for elements of the rotation set. For
example, it follows from results of Ziemian (\cite{ziemian})
that for each $r\in Int(\rot(\tphi))$,  there is a compact,
$\phi$-invariant set $Y_r\subset M$, so that $\rot(y, \tphi) = r$
for all $y\in Y_r$.  In addition,
as a consequence of theorem of Jenkinson in
\cite{jenkinson}  (see also \cite{jarekthesis} 
for each such  $r$  there is a $\phi$-invariant, ergodic, fully supported
Gibbs probability measure $\mu_{r}$ with
$\rot(x, \tphi) = r$ for $\mu_{r}$-almost every point.

In contrast, it is a simple consequence of 
Theorem~\ref{pA1} that when
$\phi_* = \id$ the rotation number doesn't exist
for the topologically typical point in $M$.
This implies that the topologically generic
point is not generic for any $\phi$-invariant Borel measure.
Here is the argument:
assume that $0 \in Int (\rot(\tphi))$. Thus from 
Theorem~\ref{pA1} there is a dense, $G_\delta$ set
$X_0\subset M$ so that $x\in X_0$ implies that
for any lift $\tx$ of $x$, the orbit $o(\tx,\tphi)$ is dense
in $\tM$. Thus for  $x\in X_0$, if $\rot(x,\tphi)$ exists, 
 it must be zero. But now
pick $\vp/q \in  Int (\rot(\tphi))$ with $0 \not = \vp/q$ 
and let $\tf = \delta_{-\vp,q}\circ \tphi^q$. Using 
\eqref{rotpower}, $0 \in Int (\rot(\tf))$ and so again
using Theorem~\ref{pA1}, there is a dense, $G_\delta$ set
$X_1\subset M$ so that $x\in X_1$ implies that
 if $\rot(x,\tf)$ exists, it must be zero and so  $\rot(x,\tphi) = \vp/q$.
It then follows that for $x$ in the dense, $G_\delta$-set 
$X_0\cap X_1$, $\rot(x, \tphi)$ cannot exist.

While there are rel \pA\ maps in every isotopy class,  
there are some classes, for example the identity class,
which cannot contain a true \pA\ map, \ie\ one whose
invariant foliations have no interior one-prong singularities. 
Thurston  observed that there are many mapping classes
which act trivially on $H_1(M)$ which do contain a true \pA\ maps, or in
the language of his classification theorem, are \pA\
mapping classes (\cite{thurston}). 
The collection of mapping classes which act trivially
on $H_1(M)$ is called the Torelli group  and its
properties have been much studied (see \cite{johnson, farb} for
surveys). 

As dynamical systems \pA\ maps which act trivially on homology
are quite interesting, and there are many tools available for their
study such as the rotation set and twisted transition matrices
(\cite{friedcohom, BB}).  They also have useful isotopy stable properties.
For example, if $g$ is any \homeo\ which is isotopic to a true \pA\ $\phi_*$
with $\phi_* = \id$, then if $\tphi$ is a transitive lift to $\tM$,
the corresponding lift $\tg$ of $g$ will always have ``well-traveled''
orbits, \ie\ orbits that repeatedly visit every fundamental
domain in $\tM$. This follows from Handel's global
shadowing theorem (\cite{handel}).

\subsection{$H_1$-transitivity}
We now give the proof of Theorem~\ref{pA2} stated in the
Introduction.

\medskip
\textbf{Proof of Theorem~\ref{pA2}:}
Fix a lift $\tphi$ of $\phi$ to $\tM$.
Assume that $\phi$ is $H_1$-transitive, and so
for some $q>0$, $\vn\in\Z^d$, we have that
$\eta:= \delta_{\vn}\circ\tphi^q$ is transitive. Now 
$\eta$ is a lift of $\tphi^q$ and so by Theorem~\ref{pA1},  
 $\rho(\phi_*^q) = 1$ and so $\rho(\phi_*) = 1$.

Now conversely, assume $\rho(\phi_*) = 1$. By Theorem~\ref{summary}(c),
$\rot_{F}(\tphi^{N(\phi_*)})$ has interior. Pick
$\vp/q\in Int(\rot_{F}(\tphi^{N(\phi_*)}))$. Now as in the
proof of Corollary~\ref{FFRmk}, 
we may write $\Z^d = W_1 \osum W'$ with
$W_1$ naturally identified with $\Z^d/F^{(kN(\phi_*))}$, for all
$k>0$.
Let $\vm\in\Z^d$ be  $\vm= (\vp, 0)$ with $0\in W'$ and so 
by \eqref{liftpower},
$\eta' :=\delta_{-\vm}\;\tphi^{qN(\phi_*)}$, 
 has $0\in Int(\rot_{F}(\eta'))$. Now
certainly $\rho(\phi_*) = 1$ implies $\rho(\eta')= 1$ and
so by Theorem~\ref{pA1}, $\eta'$ is transitive.

The implication (a) implies (c) also
follows from Theorem~\ref{pA1}.
 Now assume (c) and
for the sake of contradiction that $\rho(\phi_*) \not= 1$.
Fix a lift $\tphi$ and let $\tau$ be a twisted skew product
that corresponds to it and $\talpha$ the semiconjugacy given in 
\S\ref{semiconstruct}. 
Since $\rho(\phi_*) \not= 1$,  using Remark~\ref{MSTrk},
there is a nonzero linear functional $\Phi:\R^d\raw\R$ and
a constant $C'>0$ so  $|\Phi(\vn - \vn')| > C'$ implies 
that for any $\us, \us'\in\Sigma$,  
\begin{equation}\label{runsoff2}
 \left|\Phi\circ\pi_2(\tau^k(\us,\vn)) -
\Phi\circ\pi_2(\tau^k(\us',\vn'))\right|\raw \infty,
\end{equation} 
as $k\raw\infty$.
But if $\tL$ is a leaf of $\tcF^u$ that is dense in $\tM$,
certainly there are $(\us,\vn)$ and $(\us', \vn')$ 
with  $\talpha(\us,\vn) \in \tL$ and $\talpha(\us',\vn') \in \tL$
and  $|\Phi(\vn - \vn')| > C'$. But then by \eqref{runsoff2}
and the fact from \S\ref{quasisect} that $\talpha$ is 
a quasi-isometry from the 
the pseudometric $d_1((\us, \vn), (\us', \vn')) =
\|\vn-\vn'\|$ to a
lifted metric $\td$ on $\tM$, we have
$\td(\talpha(\tau^k(\us,\vn)), \talpha(\tau^k(\us',\vn')))
\raw\infty$. Thus since $\talpha$ is a semiconjugacy,
$\td(\tphi^k(\talpha(\us,\vn)), \tphi^k(\talpha(\us',\vn')))
\raw\infty$ as $k\raw\infty$ in contradiction to the fact that 
$\talpha(\us,\vn) $ and  $\talpha(\us',\vn')$ are on the same 
leaf of the unstable foliation (see \eqref{dynmeaning}).  \QED\

\subsection{The lifted foliations}

In the invariant foliations associated with a rel \pA\ map
on a compact surface all nontrivial leaves are dense in $M$.
For a $H_1$-transitive rel \pA\ map, for the lifted
foliations the typical leaf is dense, but there
are always non-dense leaves.

\begin{proposition}\label{folinfo}
Let $\tcF^u$ and $\tcF^s$ be the lifted foliations
to the universal Abelian cover $\tM$ of 
a rel \pA\ map $\phi:M\raw M$. 
\begin{compactenum}
\item If the lifted
foliation has one dense leaf, 
there is a dense $G_\delta$-set $Z\subset \tM$ so that
$\tx\in Z$ implies that the leaf of $\tcF^u$ containing
$\tx$ is dense in $\tM$. 
\item Each nontrivial leaf of the lifted foliations is unbounded.
\item The lifted foliations always has nontrivial leaves which are
not dense.
\end{compactenum}
\end{proposition}

\begin{proof}
The proof of (a) is a minor alteration of a standard proof:
fix a countable base $U_n$ for the topology of $\tM$.
For each $n$, let $A_n$ be all the points of $\tM$ 
which are contained in a leaf of  $\tcF^u$ which 
intersects $U_n$.  Since there is
a dense leaf by hypothesis, it follows immediately
that each $A_n$ is dense. Each $A_n$ is also open as a consequence
of our slightly peculiar definition of ``leaf'' in \S\ref{leafdef}. 
Now $\cap A_n$ is exactly all the points of $\tM$
contained in dense leaves, and by the Baire category theorem,
$\cap A_n$ is dense-$G_\delta$.

To prove (b) note that if in $\tM$ a nontrivial leaf was bounded,
by letting $\Gamma = N \Z^d$ for $N$ large
enough and $\Z^d = H_1(M)$, 
the rel \pA\ map $\tphi_\Gamma$ on the compact manifold 
$\tM/\Gamma$ would  possess a nontrivial leaf
that was not dense, a contradiction. 

The proof of (c) starts with the observation that 
if $\rho(\phi_*) \not = 1$, then by Theorem~\ref{pA2}, the lifted
foliations have no dense leaves and so (c) certainly follows.
So assume now that $\rho(\phi_*) = 1$ and so from Theorem~\ref{pA2}
we may find a $\tg$ which is a transitive lift of an iterate $g := \phi^k$.
Now assume additionally that $\phi_* = \id$ and so $g_* = \id$ as well. 
By Theorem~\ref{friedthm} there is  a periodic point
$x_0$  of $g$ such
that its lift $\tx_0$ to $\tM$ satisfies 
 $\rot(\tx_0, \tg) \in Fr(\rot(\tphi))$. 
 Using  Lemma~\ref{boundlem}
as in the proof of Theorem~\ref{pA1}, we see that if the leaf containing
$\tx_0$ (or one of its associated leaves if $\tx_0$ is a singularity) 
were dense,  $\rot(\tx_0, \tg)$ wouldn't be a boundary point of
$\rot(\tg)$, a contradiction.

In the more general situation that $\rho(\phi_*) = 1$,
use the argument in the previous  paragraph to show
that  the foliations in the Fried quotient of $\phi^{N(\phi_*)}$
have non-dense leaves  and so the lifts of these leaves to
$\tM$ are also not dense.
\QED\
\end{proof}

\begin{remark}
 Under a variety of hypotheses which include the case where 
a rel \pA\ $\phi$ is isotopic to the identity on a closed surface,
Pollicott and Sharp show in \cite{pollicottsharp} that
the transverse measures on the lifted foliations are
ergodic. 
\end{remark}

\bibliography{trans}
\bibliographystyle{alpha}

\end{document}